\documentclass{article}
\usepackage[utf8]{inputenc}
\usepackage{amsmath,amssymb,amsthm,amsfonts,latexsym,graphicx,subfigure,mathtools,enumitem}

\usepackage{url}
\usepackage{xcolor}
\usepackage{svg}

\newcommand{\R}{\mathbb{R}}
\newcommand{\RQ}{\textrm{RQ}}

\DeclareMathOperator{\id}{Id}

\DeclareMathOperator{\supp}{supp}

\DeclareMathOperator{\diag}{diag}

\theoremstyle{plain}
\newtheorem{theorem}{Theorem}[section]

\newtheorem{proposition}[theorem]{Proposition}
\newtheorem{corollary}[theorem]{Corollary}

\newtheorem{convention}{Convention}

\newcommand\ddfrac[2]{\frac{\displaystyle #1}{\displaystyle #2}}

\theoremstyle{definition}
\newtheorem{definition}[theorem]{Definition}
\newtheorem{example}[theorem]{Example}
\theoremstyle{remark}
\newtheorem{remark}[theorem]{Remark}

\usepackage{stmaryrd} 
\usepackage[affil-it]{authblk}
\usepackage{booktabs}   
\usepackage[numbers,sort&compress]{natbib}

\usepackage{comment} 

\allowdisplaybreaks

\usepackage{hyperref}

\usepackage[capitalise,noabbrev]{cleveref}

\usepackage[margin=3.5cm]{geometry}

\title{Coloring outside the lines: Spectral bounds for generalized hypergraph colorings}
\author{Lies Beers\thanks{e.g.m.beers@vu.nl} }
\author{Raffaella Mulas}
\date{}
\affil{Vrije Universiteit Amsterdam}
\begin{document}

\maketitle

\begin{abstract} 
It is known that, for an oriented hypergraph with (vertex) coloring number $\chi$ and smallest and largest normalized Laplacian eigenvalues $\lambda_1$ and $\lambda_N$, respectively, the inequality $\chi\geq (\lambda_N-\lambda_1)/\min\{\lambda_N-1,1-\lambda_1\}$ holds. We provide necessary conditions for oriented hypergraphs for which this bound is tight. Focusing on $c$-uniform unoriented hypergraphs, we then generalize the bound to the setting of \emph{$d$-proper colorings}: colorings in which no edge contains more than $d$ vertices of the same color. We also adapt our proof techniques to derive analogous spectral bounds for \emph{$d$-improper colorings} of graphs and for edge colorings of hypergraphs. Moreover, for all coloring notions considered, we provide necessary conditions under which the bound is an equality.

\vspace{0.2cm}

\end{abstract}

\section{Introduction}
\begin{figure}
    \centering
    \includegraphics[width=0.75\linewidth]{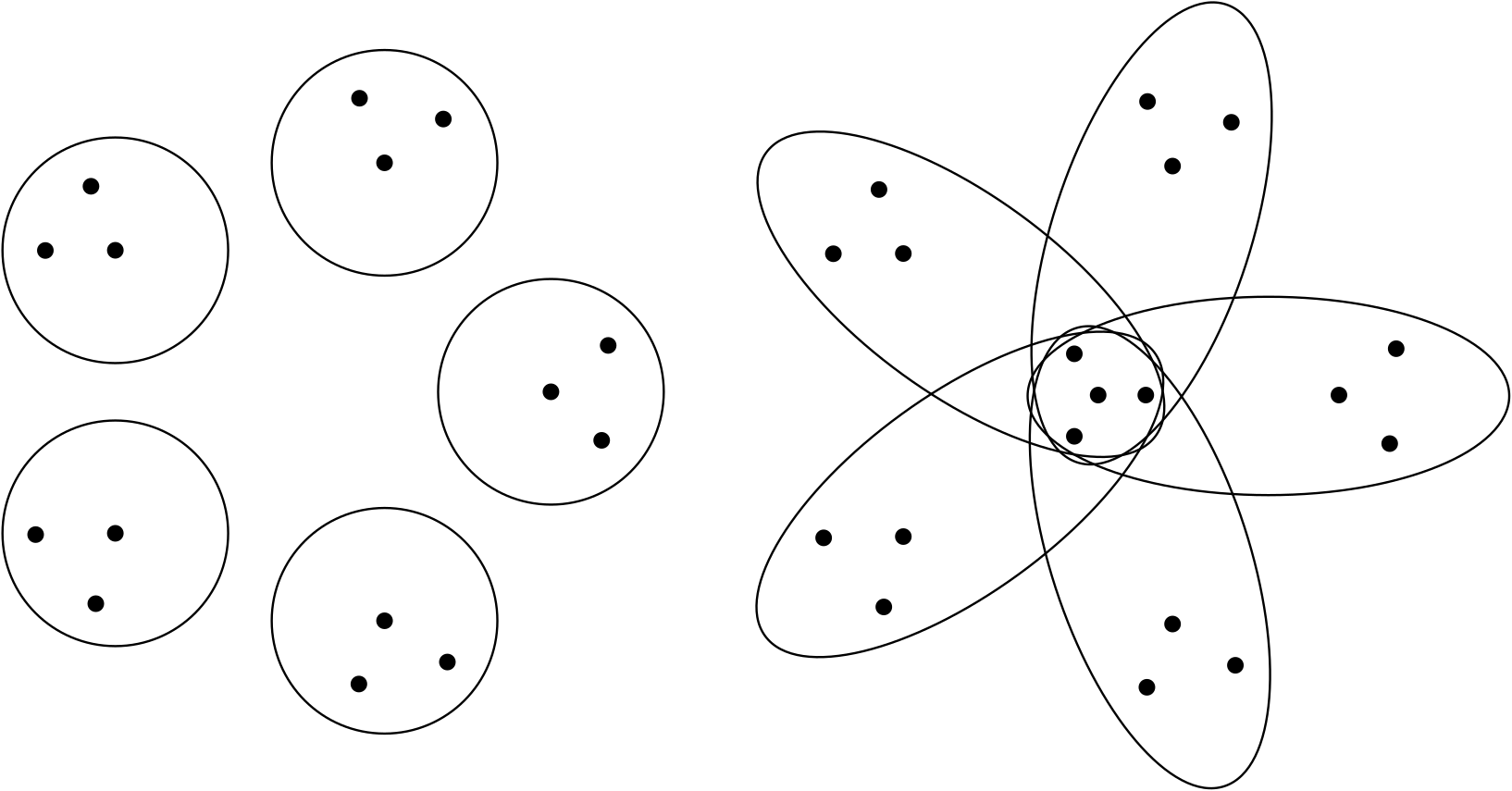}
    \caption{Two examples of hypergraphs: a disconnected hypergraph with five disjoint edges (left) and a hyperflower with five petals (right).}
    \label{fig:disjointedgestohyperflower}
\end{figure}
Classical hypergraphs, which generalize graphs by allowing edges to contain any number of vertices, were formally developed by Berge (1970) \cite{Berge1970}. Figure \ref{fig:disjointedgestohyperflower} shows two examples of hypergraphs.
In a 1994 paper \cite{Erd94}, Erd\H os wrote: \emph{As far as I know, the subject of hypergraphs was first mentioned by T.\ Gallai in a conversation with me in 1931.} \cite{chung1998erdos} Moreover, before their name was coined, hypergraphs were already being studied as \emph{set systems}, for instance, by Erd\H os and Hajnal in their 1966 work on coloring numbers of set systems \cite{ErdosHajnal1966}.
This can be translated into Berge's hypergraph framework, where the \emph{coloring number} or \emph{weak coloring number} of a hypergraph is defined as the least number of colors needed to color the vertices of a hypergraph such that no edge is \emph{monochromatic}, i.e., all edges contain at least two colors.
This notion generalizes the graph coloring number.
While the weak coloring number prohibits monochromatic edges, the \emph{strong coloring number} requires every edge to be \emph{polychromatic}, i.e., all vertices in an edge have different colors. 

To accommodate more flexible coloring constraints, De Werra (1979) \cite{DeWerra1979} introduced \emph{$I$-regular colorings}: for every edge, the number of vertices of any color that it contains is bounded both above and below.
In this paper, we introduce and study a specific case of $I$-regular colorings called \emph{$d$-proper colorings}: these are colorings where every edge contains at most $d$ vertices of a fixed color. The concept of $d$-proper colorings generalizes strong colorings (which correspond to the case $d=1$). Moreover, for \emph{$c$-uniform hypergraphs} (i.e., hypergraphs whose edges all have the same size $c$), a $c-1$-proper coloring corresponds to a weak coloring.

The concept of $d$-proper colorings was also studied by
F\"uredi and Kahn (1986) \cite{FurediKahn1986}, who gave a simultaneous upper bound on $d$ and the $d$-proper coloring number in terms of the maximum vertex degree of a hypergraph (Lemma 4.2 of \cite{FurediKahn1986}).
Moreover, $d$-proper colorings are closely related to the notion of \emph{$d$-weak independent sets}: subsets of vertices that intersect every edge in at most $d$ vertices.
In a recent preprint, Adamson, Rosenbaum, and Spirakis (2024) \cite{adamson2024distributed} studied algorithms for constructing $d$-weak independent sets of a certain size.
These connections underline the study of $d$-proper colorings from a combinatorial and algorithmic perspective.

In this paper, we shall study spectral bounds for various notions of coloring numbers, both for classical hypergraphs and \emph{oriented hypergraphs}, which were introduced by 
Shi (1992) \cite{shi1992signed}, as a generalization of classical hypergraphs in which an orientation is attached to every vertex-edge incidence.
The adjacency matrix and Kirchoff Laplacian matrices for oriented hypergraphs were introduced by Reff and Rusnak (2012) \cite{ReffRusnak2012}, and the normalized Laplacian was introduced by Jost and Mulas (2019) \cite{JostMulas2019}.
These matrices encode information about the structure of the hypergraphs, as they do in the graph case.
We refer the reader to
\cite{ChenLiuRobinsonRusnakWang2018, ChenRaoRusnakYang2015, DuttweilerReff2019, RusnakReynesJohnsonYe2021, GrillietteReynesRusnak2022, RusnakReynesLiYanYu2024, KitouniReff2019, Reff2014, Reff2016, ReffRusnak2012, RusnakRobinsonSchmidtShroff2019, Rusnak2013}
for literature on the adjacency and Kirchhhoff Laplacian matrices, and to
\cite{AndreottiMulas2022, JostMulas2019, Mulas2021, MulasZhang2021, mulas2022hypergraphs}
for literature on the normalized Laplacian matrix.

Strong colorings of a hypergraph can be obtained by coloring its \emph{representative graph}, in which every edge of size $c$ is replaced by a $c$-clique. As a result, strong colorings have received less attention in the literature.
When considering matrices associated to a hypergraph whose spectra provide different information than spectra of matrices associated to the corresponding representative graph, it becomes natural to study their spectral connection to the strong coloring number.

Abiad, Mulas and Zhang (2021) \cite{Abiad20} gave a spectral lower bound for the strong coloring number $\chi$ of an oriented hypergraph:
\begin{equation}\label{eq:boundgen}
\chi\geq\frac{\lambda_N-\lambda_1}{\min\bigl\{\lambda_N-1,1-\lambda_1\bigr\}},
\end{equation}
where $\lambda_1$ and $\lambda_N$ are the smallest and largest normalized Laplacian eigenvalues, respectively.
This bound generalizes the graph bound
\begin{equation}
\chi\geq \frac{\lambda_N}{\lambda_N-1},
\end{equation}
first proven by Elphick and Wocjan (2015) \cite{ElphickWocjan2015} and later explicitly stated by Coutinho, Grandsire, and Passos (2019) \cite{Gabriel-colouring}.

For regular graphs, this bound coincides with the \emph{Hoffman bound} \cite{Hoffman1970}, which relates the adjacency eigenvalues to the graph vertex coloring number. The Hoffman bound was generalized to hypergraphs in different ways, first by Nikiforov (2019) \cite{nikiforov2019hoffman}, who generalized the bound for $(2c)$-uniform hypergraphs using the spectrum of the quadratic form of the adjacency matrix. Banerjee (2021) \cite{Banerjee2021} generalized the Hoffman bound to a generalized hypergraph adjacency matrix for arbitrary unoriented hypergraphs.

Bilu (2006) \cite{Bilu2006} generalized the Hoffman bound to \emph{$d$-improper colorings} of graphs: vertex colorings of a graph such that no vertex has more than $d$ neighbors of the same color. Bilu proved a generalized Hoffman bound for the \emph{$d$-improper coloring number} $\chi^d$ in terms of the adjacency eigenvalues. We shall prove an analogous bound for the normalized Laplacian eigenvalues (Theorem \ref{thm:d-improper}).

Building on both Bilu's bound and the bound in \eqref{eq:boundgen}, our contributions include the following:
\begin{itemize}
    \item We present an alternative proof of the bound in \eqref{eq:boundgen} and prove necessary conditions for the bound to be sharp, which gives a relation between the eigenspaces of the smallest and largest eigenvalues (Propositions \ref{prop:1-lambda1} and \ref{prop:lambdaN-1}).
    \item Focusing on $c$-uniform hypergraphs that are \emph{unoriented}, i.e., for which all vertex-edge incidences have a positive orientation, the bound is \begin{equation}\label{eq:boundc-un}
        \chi \geq \frac{c-\lambda_1}{1-\lambda_1}.
    \end{equation}
    For $c$-uniform unoriented hypergraphs for which the bound is an equality, we prove more explicit results about their structure (Theorem \ref{thm:uniformstrong} and Corollary \ref{cor:equality}).
    \item We generalize the above results to the $d$-proper coloring number $\chi_d$ (the minimal number of colors needed to color a hypergraph $d$-properly), and we prove the bound (Theorem \ref{thm:d-proper})
    \begin{equation}\label{eq:d-proper}
        \chi_d \geq\frac{c-\lambda_1}{d-\lambda_1}.
    \end{equation}
    \item We use the proof idea of Theorem \ref{thm:d-proper} to prove a bound similar to \eqref{eq:d-proper} for $d$-improper colorings of graphs (Theorem \ref{thm:d-improper}).
    \item We use the idea of the proof of Theorem \ref{thm:cor5.4AMZ} to prove a bound similar to the one in \eqref{eq:boundgen} for edge coloring of hypergraphs (Theorem \ref{thm:edgecoloring}).
\end{itemize}




\section{Background}

\subsection{The normalized Laplacian for hypergraphs}
We begin by defining the normalized Laplacian for graphs, originally introduced by Chung (1992) \cite{chung}, using the vertex-edge incidence matrix. We then extend this definition to hypergraphs, following the generalization proposed by Jost and Mulas (2019) \cite{JostMulas2019}. We emphasize properties that are most relevant to our paper, many of which are also discussed in \cite{JostMulas2019,Mulas2021,mulas2022hypergraphs,AndreottiMulas2022,Mulas-Cheeger}.
Specifically, we define oriented hypergraphs, introduce the Rayleigh quotients of their normalized Laplacian, and explore the relationship between the Rayleigh quotient and the eigenvalues of the normalized Laplacian.
\begin{definition}\label{def:orientedhypergph}
    An \emph{oriented hypergraph} $\Gamma = (V, E,\varphi)$ consists of a vertex set $V$, an edge set $E \subset \mathcal P\bigl(V\bigr)$, and an orientation function $\varphi\colon V\times E \to \{0,-1,1\}$ such that $\varphi(v,e)\neq 0 \iff v\in e$.
\end{definition}

We consider the following subclass of oriented hypergraphs.
\begin{definition}
    For $c\geq 1$, an oriented hypergraph is \emph{$c$-uniform} if $|e| = c$ for all $e\in E$.
\end{definition}
\begin{example}
    An example of a $3$-uniform oriented hypergraph can be found in Figure \ref{fig:hypergraphA=0} below. 
\end{example}
\begin{definition}
    Let $v,v'\in V$ and $e\in E$.
    If $\varphi(v,e) = 1$ (resp.\ $\varphi(v,e)=-1$), then $v$ is an \emph{output} (resp.\ \emph{input}) of $e$. If 
    \begin{equation*}
        \varphi(v,e) = \varphi(v',e)\neq 0 \quad \text{(resp.\ $\varphi(v,e) = -\varphi(v',e)\neq 0)$,}
    \end{equation*}
    then $v$ and $v'$ are \emph{co-oriented} (resp.\ \emph{anti-oriented}) in $e$.
\end{definition}
\begin{remark}
    A simple graph can be seen as an oriented hypergraph such that every edge has exactly one input and one output.
\end{remark}
We now give a generalization of connected and bipartite graphs, respectively.
\begin{definition}\label{def:bipartite}
    The oriented hypergraph $\Gamma = (V, E,\varphi)$ is \emph{bipartite} if we can partition its vertex set $V= V_1\sqcup V_2$ such that, for every edge $e\in E$, either $e$ has all its inputs in $V_1$ and all its outputs in $V_2$, or vice versa.
\end{definition}
\begin{example}
    If an oriented hypergraph $\Gamma = (V, E,\varphi)$ is such that $\varphi(V) = \{0,1\}$, then $\Gamma$ is bipartite.
\end{example}
This leads us to the following definition.
\begin{definition}\label{def:underlying}
    For an oriented hypergraph $\Gamma = (V, E,\varphi)$, its \emph{underlying hypergraph} is $\Gamma^{+}\coloneqq \bigl(V,E,\bigl|\varphi\bigr|\bigr)$, where $|\varphi|$ indicates whether a vertex belongs to an edge, ignoring the orientation. That is, $$\bigl|\varphi\bigr|(v,e):=|\varphi(v,e)|.$$ 
\end{definition} 
For the remainder of this section, we fix an oriented hypergraph $\Gamma = (V, E, \varphi)$ with $|V| = N$ and $|E| = M$. Using the orientation function $\varphi$, we define the incidence matrix, the adjacency matrix, and the normalized Laplacian of $\Gamma$.
\begin{definition}
    The \emph{incidence matrix} $\mathcal I$ is the $N\times M$ matrix with entries
    \[
    \mathcal I_{i,j} = \varphi\bigl(v_i,e_j\bigr).
    \]
\end{definition}
Note that, in the graph case, the Kirchhoff Laplacian of a graph $G$ is defined as the matrix $$K\coloneqq K(G) \coloneqq \mathcal I \mathcal I^\top,$$ where the incidence matrix corresponds to any arbitrary orientation of the graph, such that each edge has exactly one input and one output. Equivalently, the graph Kirchhoff Laplacian can be written as $K = D - A$, where $D$ is the \emph{degree matrix}, whose diagonal entries are the degrees of the corresponding vertices, and $A$ is the adjacency matrix. Similarly, the \emph{normalized Laplacian} of a graph $G$ is the matrix $$L \coloneqq L(G) \coloneqq D^{-1}\mathcal I \mathcal I^\top = \id - D^{-1}A.$$

We define the adjacency matrix, degree (matrix), Kirchhoff Laplacian, and normalized Laplacian for hypergraphs in analogy with the graph case.
\begin{definition}\label{def:matrices}
    The \emph{Kirchhoff Laplacian} is the matrix
    \begin{equation*}
        K\coloneqq \mathcal I \mathcal I^\top.
    \end{equation*}
    Additionally, the \emph{degree matrix} is the diagonal matrix 
    \[D\coloneqq \diag\bigl(\deg(v_1),\ldots,\deg(v_N)\bigr),\] where the degree of a vertex $v$ is defined as
    \begin{equation*}
        \deg(v)\coloneqq \bigl| \bigl\{ e\in E \colon v\in e \bigr\} \bigr|.
    \end{equation*}
    The \emph{adjacency matrix} is the $N\times N$ matrix with zeroes along the diagonal such that, for $i\neq j$, its entries are given by
    \begin{align*}
    A_{i,j} \coloneqq &\# \text{ edges where $v_i$ and $v_j$ are anti-oriented }\\ - &\# \text{ edges where $v_i$ and $v_j$ are co-oriented.}
    \end{align*}
    Finally, the \emph{normalized Laplacian} is the $N\times N$ matrix defined as
    \[
    L \coloneqq D^{-1}K = \id - D^{-1}A.
    \]
    We denote its eigenvalues by
    \[
    \lambda_1\leq \ldots \leq \lambda_N.
    \]
\end{definition}
We use the notation
\[
\bigl\{ \mu_1^{(m_1)},\ldots,\mu_p^{(m_p)} \bigr\}
\]
to denote a multiset containing the element $\mu_i$ with multiplicity $m_i$. Throughout the paper, we shall also use the notation $m_\Gamma(\lambda)$ for the multiplicity of $\lambda$ as an eigenvalue of the normalized Laplacian of $\Gamma$.
\begin{remark}\label{rmk:spectrumconcomp}
    The spectrum of a hypergraph equals the union of the spectra of its connected components.
\end{remark}
\begin{remark}
The entries of $L$ can be written as
\[
L_{i,j} = \frac{\text{ edges where $v_i$ and $v_j$ are co-oriented } - \# \text{ edges where $v_i$ and $v_j$ are anti-oriented}}{\deg v_i}.
\]
Note that there is not necessarily a one-to-one correspondence between an oriented hypergraph and its normalized Laplacian. For example, a hypergraph with no edges satisfies $A=\boldsymbol{0}$ and $L=\id$. The following example gives a non-trivial example of an oriented hypergraph with $A = \boldsymbol{0}$ and $L = \id$.
\end{remark}
\begin{example}\label{ex:A=0}
    Let $\Gamma$ (Figure \ref{fig:hypergraphA=0}) be the hypergraph with vertex set $V = \bigl\{v_1,v_2,v_3,v_4\bigr\}$, edge set $E = \bigl\{ e_1,e_2,e_3,e_4\bigr\}$, with $e_i\coloneqq V\setminus \{v_i\}$, and incidence matrix
    \[
    \mathcal I(\Gamma) \coloneqq \begin{bmatrix}
        0&1&-1&-1\\
        1&0&-1&1\\
        1&-1&0&-1\\
        1&1&1&0
    \end{bmatrix}.
    \]
    Then, one can verify that $A = 3\id - \mathcal I(\Gamma)\mathcal I(\Gamma)^\top =  \boldsymbol{0}$, hence $L = \id$.
\end{example}
\begin{figure}
    \centering
    \includegraphics[width=0.35\linewidth]{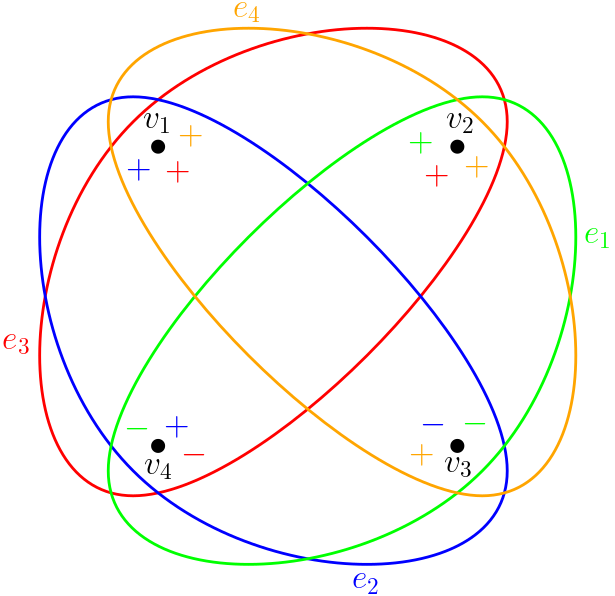}
    \caption{An oriented hypergraph with $A=\boldsymbol 0$.}
    \label{fig:hypergraphA=0}
\end{figure}
The above example motivates the following observation.
\begin{remark}\label{rmk:A=0}
We have 
\begin{align*}
    A=\textbf{0} &\iff L=I \\
    &\iff \lambda_1=\ldots=\lambda_N=1\\
    &\iff \lambda_1=1\\
    &\iff \lambda_N=1.
\end{align*}
The last two characterizations follow from the fact that $\sum_{i=1}^N\lambda_i=N$. Hence, if $$1=\lambda_1\leq \lambda_2\leq \ldots\leq\lambda_N,$$ then $1=\lambda_1= \ldots=\lambda_N$. Similarly, if $$\lambda_1\leq \lambda_2\leq \ldots\leq\lambda_N=1,$$ then $\lambda_1= \ldots=\lambda_N=1$. In particular, this implies that, if $A\neq \textbf{0}$, then $\lambda_1<1$ and $\lambda_N>1$.
\end{remark}

We now give a characterization of the eigenvalues of the normalized Laplacian for hypergraphs, using Rayleigh quotients. Chung (1992) \cite{chung} gave a similar characterization for the graph case, and this was generalized to hypergraphs by Jost and Mulas (2019) \cite{JostMulas2019}. We will follow \cite{JostMulas2019} in the remainder of this subsection.\\

We first define the following scalar product.
\begin{definition}
    Let $C(V)$ denote the vector space of functions $f:V\rightarrow\mathbb{R}$ and, given $f,g\in C(V)$, let
	\begin{equation*}
	\langle f,g\rangle\coloneqq\sum_{v\in V}\deg v\cdot f(v)\cdot g(v).
	\end{equation*}
\end{definition}
We can see the normalized Laplacian $L$ as an operator $C(V)\rightarrow C(V)$ such that
\begin{equation}\label{eq:Lf}
    Lf(v_i)=f(v_i)-\frac{1}{\deg v_i}\sum_{j\neq i}A_{ij}f(v_j).
\end{equation}
Note that $L$ is self-adjoint with respect to $\langle \cdot, \cdot\rangle$, and that $\langle \cdot, \cdot \rangle$ is positive definite.\\

With the Courant-Fischer-Weyl min-max Principle below, we can characterize the eigenvalues of $L$.

\begin{theorem}[Courant-Fischer-Weyl min-max Principle]
\label{min-max theorem}

Let $H$ be an $N$-dimensional vector space with a positive definite scalar product $(\cdot,\cdot)$, and let $A\colon H\rightarrow H$ be a self-adjoint linear operator. Let $\mathcal{H}_{k}$  be the  family of all $k$-dimensional subspaces of $H$.
  Then the eigenvalues $\lambda_{1}\leq \ldots \leq \lambda_{N}$ of
$A$ can be obtained by 
\begin{equation}
\label{min-max}
\lambda_{k}=\min_{H_k\in \mathcal{H}_{k}}\max_{g(\neq 0)\in H_{k}}\frac{(Ag,g)}{(g,g)}=\max_{{H}_{N-k+1} \in \mathcal{H}_{N-k+1}}\min_{g(\neq 0)\in {H}_{N-k+1}}\frac{(Ag,g)}{(g,g)}.
\end{equation}
The vectors $g_k$ realizing such a min-max or max-min then are corresponding
eigenvectors, and the min-max spaces $H_{k}$ are spanned
by the eigenvectors for the eigenvalues $\lambda_1,\dots
,\lambda_k$, and analogously, the max-min spaces $H_{N-k+1}$ are spanned
by the eigenvectors for the eigenvalues $\lambda_k, \dots, \lambda_N$.

Thus, we also have

\begin{equation}
\label{min-max1}
\lambda_{k}=\min_{g(\neq 0)\in H, (g,g_j)=0\ \mathrm{ for }\ j=1,\dots ,k-1}\frac{(Ag,g)}{(g,g)}=\max_{g(\neq 0)\in H, (g,g_\ell)=0\ \mathrm{ for }\ \ell=k+1,\dots ,N}\frac{(Ag,g)}{(g,g)}.
\end{equation}

In particular, 
\begin{equation}
\label{min-max2}
\lambda_{1}=\min_{g(\neq 0)\in H}\frac{(Ag,g)}{(g,g)},\qquad \lambda_N=\max_{g(\neq 0)\in H}\frac{(Ag,g)}{(g,g)}.
\end{equation}
\end{theorem}
\begin{definition}\label{RQ}
$(Ag,g)/(g,g)$ is called the \emph{Rayleigh quotient} of $g$. 
\end{definition}

Now, for $f\in C(V)$, we let
\begin{equation*}
    \RQ(f) \coloneqq \frac{\langle Lf,f\rangle}{\langle f,f \rangle} = \frac{\sum_{e\in E}\biggl(\sum_{v\in e \text{ input}} f(v) - \sum_{w\in e \text{ output}}f(w)\biggr)^2}{\sum_{v\in V}\deg v f(v)^2}.
\end{equation*}
We can easily see from the expression of the Rayleigh quotient $\RQ(f)$, in combination with Equation \eqref{min-max2}, that
\[
\lambda_1 = \min_{f\neq 0}\RQ(f) \geq 0.
\]
Chung (1992) \cite{chung} showed that $\lambda_1=0$ in the graph case. However, in the hypergraph case, this is not necessarily true. For example, we saw in Remark \ref{rmk:A=0} that $\lambda_1$ can be equal to $1$.

Furthermore, in the graph case, we have that the largest eigenvalue satisfies $\lambda_N \leq 2$ \cite{chung}. This was generalized to hypergraphs by Mulas (2021) \cite{Mulas2021}, as follows.
\begin{proposition}\label{prop:eigenvalues}
    For every connected, oriented hypergraph $\Gamma$,
    \begin{align*}
        \lambda_N \leq \max_{e\in E}|e|.
    \end{align*}
    Furthermore, the inequality is sharp if and only if $\Gamma$ is bipartite (Definition \ref{def:bipartite}) and $\Gamma$ is a $c$-uniform hypergraph.
\end{proposition}
Mulas et al.\ (2022) \cite{mulas2022hypergraphs} also proved the following proposition, relating the normalized Laplacian spectrum of a bipartite hypergraph (Definition \ref{def:bipartite}) and its underlying hypergraph (Definition \ref{def:underlying}).
\begin{proposition}
    If $\Gamma$ is bipartite, then it is isospectral to its underlying hypergraph.
\end{proposition}
\begin{remark}\label{rmk:lambdaN=c}
    If $\Gamma$ is an oriented hypergraph such that all vertices are inputs in all edges they belong to (or equivalently, $\Gamma$ is the underlying hypergraph of some oriented hypergraph), then by Proposition \ref{prop:eigenvalues} its largest eigenvalue $\lambda_N$ equals $\max_{e\in E}|e|$ if and only if $\Gamma$ is $c$-uniform. In this case, as shown in \cite{Mulas2021}, the multiplicity of the eigenvalue $\lambda_N=c$ equals the number of connected components, and the corresponding eigenfunctions are those functions that are constant on any connected component.
\end{remark}

\subsection{Coloring hypergraphs}\label{section:overview-colorings}
We now provide an overview of the different notions of coloring considered in this paper, most of which were introduced earlier in the introduction. For each notion, the corresponding coloring number is defined as the minimum number of colors needed to color the hypergraph  according to that notion.

We first define a proper vertex coloring for a graph, and then generalize this to hypergraphs in different ways.
\begin{definition}
    Let $G=(V,E)$ be a graph. A \emph{(vertex) $k$-coloring} is a function $c:V\to \{1,\ldots,k\}$, and it is \emph{proper} if $v\sim w$ (i.e., $\{v,w\}\in E$) implies that $c(v)\neq c(w)$. The corresponding coloring number is denoted by $\chi\coloneqq \chi(G)$.
\end{definition}
We now define a proper strong $k$-coloring for the vertices of an oriented hypergraph.
\begin{definition}\label{def:coloring1}
    Let $\Gamma$ be an oriented hypergraph. A \emph{proper strong (vertex) $k$-coloring} is a function $c\colon V(\Gamma)\to \{1,\ldots,k\}$ such that, for every edge $e$ and any two vertices $v_1,v_2\in e$, we have $c(v_1)\neq c(v_2)$. The corresponding coloring number is denoted by $\chi\coloneqq \chi(\Gamma)$.
\end{definition}
We shall consider the notion of strong coloring in Section \ref{sec:strong}, while in Section \ref{sec:d-proper} we shall consider the following generalization.
\begin{definition}\label{def:d-proper1}
    Let $\Gamma$ be an oriented hypergraph and let $d\geq 1$. A (vertex) $k$-coloring $c\colon V(\Gamma)\to \{1,\ldots,k\}$ is \emph{$d$-proper} if every edge contains at most $d$ vertices of every color. The corresponding coloring number is denoted by $\chi_d\coloneqq \chi_d(\Gamma)$.
\end{definition}

In particular, a proper strong $k$-coloring is a $d$-proper $k$-coloring with $d=1$.\\

We also introduce and study the concept of $q$-tailored coloring, which will be the focus of Subsection \ref{subsec:q-tailored}. In the setting of $c$-uniform hypergraphs in which all vertices are inputs in all edges they belong to, $q$-tailored colorings generalize $d$-proper colorings. 
\begin{definition}\label{def:q-tailored1}
    Let $\Gamma$ be a $c$-uniform hypergraph in which all vertices are inputs in all edges they belong to, and let $q\in [0,c-1]$. A (vertex) $k$-coloring $c\colon V(\Gamma)\to \{1,\ldots,k\}$ that has coloring classes $V_1,\ldots,V_k$ is \emph{$q$-tailored} if, for all $i=1,\ldots,k$ and for all $v\in V_i$, we have
    \begin{equation}
        \sum_{w\in V_i}|A_{v,w}|\leq q \deg v.
    \end{equation}
    The corresponding coloring number is denoted by $\chi^{q\text{-t}}\coloneqq \chi^{q\text{-t}}(\Gamma)$.
\end{definition}

In particular, if $\Gamma$ is a $c$-uniform hypergraph in which every vertex is an input in all edges it belongs to, then any $d$-proper coloring is also a $(d-1)$-tailored coloring. This is why $q$-tailored colorings generalize $d$-proper colorings in this setting.\\

Informally, a $q$-tailored coloring ensures that no vertex is too connected to others of the same color: the average number of same-colored neighbors of any vertex $v\in V$ in an edge is at most a fraction $q$ of its degree.\\

Furthermore, in Section \ref{sec:d-improper}, we shall focus on the following generalization of the strong coloring number for graphs.
\begin{definition}\label{def:d-improper1}
    A vertex coloring $c\colon V\to \{1,\ldots,k\}$ is \emph{$d$-improper} if every vertex $v$ has at most $d$ neighbors with the same color as $v$. The \emph{$d$-improper coloring number} of $G$, denoted $\chi^d=\chi^d(G)$, is the minimal number of colors needed to color the vertices $d$-improperly. The corresponding coloring number is denoted by $\chi^d\coloneqq \chi^d(G)$.
\end{definition}
Finally, we shall consider edge colorings of hypergraphs, which are defined as follows and will be the focus of Section \ref{sec:edgecol}.
\begin{definition}\label{def:edge-coloring}
    Let $\Gamma=(V,E,\varphi)$ be an oriented hypergraph. A \emph{proper strong $k$-edge-coloring} is a function $c'\colon E\to \{1,\ldots,k\}$ such that $c'(e_1) = c'(e_2)\implies e_1\cap e_2=\varnothing$. The \emph{strong edge coloring number} $\chi'\coloneqq \chi'(\Gamma)$ is the minimum number of colors needed to strongly color the edges properly.
\end{definition}


\section{Eigenvalue bounds for strong hypergraph vertex coloring}\label{sec:strong}
\subsection{Strong vertex coloring of general oriented hypergraphs}
We again fix an oriented hypergraph $\Gamma = (V, E, \varphi)$.
We first present an alternative proof of the following result, which was first proven by Abiad, Mulas and Zhang (2020) \cite[Corollary 5.4]{Abiad20}. 
\begin{theorem}\label{thm:cor5.4AMZ}
    Let $\Gamma$ be an oriented hypergraph and let $\chi(\Gamma)$ be defined as in Definition \ref{def:coloring1}. Then, we have that
    \[
    \chi(\Gamma) \geq \frac{\lambda_N - \lambda_1}{\min\bigl\{ \lambda_N-1, 1-\lambda_1 \bigr\}}.
    \]
\end{theorem}
Our proof relies on Rayleigh quotients and leads to new consequences, which are detailed in Proposition \ref{prop:1-lambda1} and \ref{prop:lambdaN-1}. These results show that, if equality holds in the bound from Theorem \ref{thm:cor5.4AMZ}, then there must be a specific relation between the eigenspaces corresponding to $\lambda_1$ and $\lambda_N$. To better understand the structure of such hypergraphs, one needs to know these eigenspaces. This is why, to gain further insight, in Subsection \ref{sec:uniformhypergraphs}, we shall focus on $c$-uniform unoriented hypergraphs, for which the eigenspace associated with $\lambda_N=c$ is known to consist of constant functions.\\

We need the following:
\begin{convention}\label{conv:1}
We recall that, by Remark \ref{rmk:A=0}, we have $\lambda_1=1$ if and only if $\lambda_N=1$. We use the convention that, in this case, 

\[
\frac{\lambda_N-\lambda_1}{\lambda_N-1} = \frac{\lambda_N-\lambda_1}{1-\lambda_1} = 1.
\]
\end{convention}

\begin{proof}[Alternative proof of Theorem \ref{thm:cor5.4AMZ}]
    We first consider the case where $\lambda_1 = \lambda_N = 1$. By Convention \ref{conv:1}, in this case, the inequality in the statement becomes
    \[
    \chi(\Gamma)\geq 1,
    \]
    which is trivially true.

    Now assume that $\lambda_1<1$ and hence, by Remark \ref{rmk:A=0}, $\lambda_N > 1$.
    We first show that $$\chi(\Gamma)\geq \frac{\lambda_N-\lambda_1}{1-\lambda_1}.$$ Fix a proper $\chi$-coloring, and let $V_1,\ldots,V_\chi$ be the corresponding coloring classes. Furthermore, let $g$ be an eigenfunction for $\lambda_N$, so that $\RQ(g)=\lambda_N$.
    Let also $l\leq \chi$ be such that
    \[
    \supp(g)\subseteq \bigcup_{i=1}^lV_i,
    \]
    where we rearrange the coloring classes if necessary, to ensure that $l$ is as small as possible. Note that $l\geq2$, because, if $\supp(g)\subseteq V_1$, then $\RQ(g)= 1$.

    Now define, for $1\leq i<j\leq  l$,
    \begin{equation}\label{eq:S_ij}
    S_{ij}^g\coloneqq  \sum_{v_i\in V_i, v_j\in V_j}A_{v_i,v_j}g(v_i)g(v_j).
    \end{equation}
    We have that
    \begin{align*}
        \lambda_N &= \RQ(g) \\&= 1 - 2\ddfrac{\sum_{v,w\in V}A_{v,w}\cdot g(v)\cdot g(w)}{\sum_{v\in V}\deg v\cdot g(v)^2}\\
        &= 1 - 2\ddfrac{\sum_{1\leq i<j\leq  l}\sum_{v_i\in V_i,v_j\in V_j}A_{v_i,v_j}\cdot g(v_i)\cdot g(v_j)}{\frac1{ l-1}\sum_{1\leq i< j\leq  l}\sum_{v\in V_i\cup V_j}\deg v \cdot g(v)^2}\\
        &\leq 1 + 2\ddfrac{\sum_{1\leq i<j\leq  l}\max\biggl\{-S^g_{ij},0\biggr\}}{\frac1{ l-1}\sum_{1\leq i< j\leq  l}\sum_{v\in V_i\cup V_j}\deg v \cdot g(v)^2}.
    \end{align*}
    Note that, since  $\lambda_N>1$, we must have that $S^g_{ij}< 0$ for some $i,j$. Now we let, for all $i,j$,
     \begin{equation}\label{eq:g_ij}
     g_{ij}(v)\coloneqq \begin{cases}
        g(v), &\text{ if } v\in V_i,\\
        -g(v), &\text{ if } v\in V_j,\\
        0, &\text{ otherwise.}
    \end{cases}
    \end{equation}
    and we assume without loss of generality that
    \[
    \RQ(g_{12}) = \min_{1\leq i<j\leq  l} \RQ( g_{ij}).
    \]
    Note that $S_{12}^g<0$.
    Then,
    \begin{align*}
        \lambda_N &\leq 1 + 2\ddfrac{\sum_{1\leq i<j\leq  l}\max\biggl\{-S^g_{ij},0\biggr\}}{\frac1{ l-1}\sum_{1\leq i< j\leq  l}\sum_{v\in V_i\cup V_j}\deg v \cdot g(v)^2}\\
        &\leq 1 + 2( l-1)\ddfrac{\bigl|S^g_{12}\bigr|}{\sum_{v\in V_1\cup V_2}\deg v \cdot g(v)^2}\\
        &\overset{(i)}{=} 1+\bigl( l-1\bigr)\bigl(1-\RQ(g_{12})\bigr).
    \end{align*}
    To prove $(i)$, recall that $S_{12}^g<0$, and note that
    \begin{align*}
        \RQ\bigl(g_{12}\bigr) &= 1 -2\frac{\sum_{\substack{v\in V_1\\ w\in V_2}}A_{v,w}\cdot g(v)\cdot -g(w)}{\sum_{v\in V_1\cup V_2}\deg v\cdot g(v)^2}\\
        &= 1-2\frac{\bigl|S_{12}^g\bigr|}{\sum_{v\in V_1\cup V_2}\deg v\cdot g(v)^2}.
    \end{align*}
    From the equality $\lambda_N\leq 1+(l-1)(1-\RQ(g_{12}))$, it follows that
    \begin{equation*}
        \lambda_N \leq 1+\bigl( l-1\bigr)\bigl(1-\lambda_1\bigr),
    \end{equation*}
     which is equivalent to
    \begin{equation}\label{eq:inequality1.1}
        \frac{\lambda_N-\lambda_1}{1-\lambda_1} \leq l \leq \chi.
    \end{equation}
    The proof of the inequality
    \[
    \chi(\Gamma) \geq \frac{\lambda_N-\lambda_1}{\lambda_N-1}
    \]
    is analogous. In particular, let $h$ be an eigenfunction such that $\RQ(h) = \lambda_1$, and let $k\leq \chi$ be such that
    \[
    \supp(h)\subseteq \bigcup_{i=1}^kV_i,
    \]
    where we rearrange the coloring classes if necessary, to ensure that $k$ is as small as possible. Note that $k\geq2$, because, if $\supp(h)\subseteq V_1$, then $\RQ(h)= 1$.

    We define, for all $1\leq i<j\leq k$, the functions
    \begin{equation}\label{eq:h_ij}
     h_{ij}(v)\coloneqq \begin{cases}
        h(v), &\text{ if } v\in V_i,\\
        - h(v), &\text{ if } v\in V_j,\\
        0, &\text{ otherwise.}
    \end{cases}
    \end{equation}
    We assume without loss of generality that
    \[
    \RQ( h_{12}) = \max_{1\leq i<j\leq k}\RQ(h_{ij}),
    \]
    and we let $S^h_{ij}$ be defined as in \cref{eq:S_ij}. We see that
    \begin{align*}
        \lambda_1 &= \RQ(h) \\&= 1 - 2\ddfrac{\sum_{v,w\in V}A_{v,w}\cdot h(v)\cdot h(w)}{\sum_{v\in V}\deg v\cdot h(v)^2}\\
        &= 1 - 2\ddfrac{\sum_{1\leq i<j\leq  k}\sum_{v_i\in V_i,v_j\in V_j}A_{v_i,v_j}\cdot h(v_i)\cdot h(v_j)}{\frac1{ k-1}\sum_{1\leq i< j\leq  k}\sum_{v\in V_i\cup V_j}\deg v \cdot h(v)^2}\\
        &\geq 1 - 2\ddfrac{\sum_{1\leq i<j\leq  k}\max\bigl\{S^h_{ij},0\bigr\}}{\frac1{ k-1}\sum_{1\leq i< j\leq  k}\sum_{v\in V_i\cup V_j}\deg v \cdot h(v)^2}\\
        &\geq 1 - 2( k-1)\ddfrac{S_{12}^h}{\sum_{v\in V_1\cup V_2}\deg v \cdot h(v)^2}\\
        &\overset{(ii)}{=} 1-\bigl( k-1\bigr)\bigl(\RQ(h_{12})-1\bigr).
    \end{align*}     
        
    To prove $(ii)$, we use the fact that $S_{12}^h>0$ (since $\RQ(h)<1$), and that
    \begin{align*}
        \RQ\bigl(h_{12}\bigr) &= 1 -2\frac{\sum_{\substack{v\in V_1\\ w\in V_2}}A_{v,w}\cdot h(v)\cdot -h(w)}{\sum_{v\in V_1\cup V_2}\deg v\cdot h(v)^2}\\
        &= 1+2\frac{S_{12}^h}{\sum_{v\in V_1\cup V_2}\deg v\cdot h(v)^2}.
    \end{align*}
    By assumption, it now follows that
    \begin{equation*}
    \begin{aligned}
        \lambda_1&\geq 1-( k-1)(\lambda_N-1) \\
        &=  k(1-\lambda_N)+\lambda_N,
    \end{aligned}
    \end{equation*}
    i.e.,
    \begin{equation}\label{eq:inequality2.2}
        \chi \geq k \geq \frac{\lambda_N-\lambda_1}{\lambda_N-1}.
    \end{equation}
    This concludes the proof.
\end{proof}

Note that we can rewrite the inequality in \eqref{eq:inequality1.1} to see that
\begin{equation}
\lambda_1 \leq \frac{\chi-\lambda_N}{\chi-1}.
\end{equation}
Similarly, we can rewrite the inequality in \eqref{eq:inequality2.2} to obtain
\begin{equation}
    \lambda_N\geq \frac{\chi-\lambda_1}{\chi-1}.
\end{equation}
The following two propositions are immediate consequences of the proof of Theorem \ref{thm:cor5.4AMZ}. They show that, if either inequality \eqref{eq:inequality1.1} or \eqref{eq:inequality2.2} is an equality for a given hypergraph, then there exists a specific relation between the eigenspaces of $\lambda_1$ and $\lambda_N$. Moreover, in such cases, at least one of the eigenvalues $\lambda_1$ or $\lambda_N$ must have multiplicity at least $\chi- 1$.
\begin{proposition}\label{prop:1-lambda1}
    Let $\Gamma$ be an oriented hypergraph such that $A(\Gamma)\neq \boldsymbol 0$ and $$\chi=\frac{\lambda_N-\lambda_1}{1-\lambda_1}.$$ Then, the following hold:
    \begin{enumerate}
    \item For every eigenfunction $g$ with eigenvalue $\lambda_N$, and for all proper $\chi$-colorings with corresponding coloring classes $V_1,\ldots,V_\chi$, we have
    \[
    \supp(g) \cap V_i \neq \varnothing \text{ for all } 1\leq i\leq \chi.
    \]
    \item For every eigenfunction $g$ with eigenvalue $\lambda_N$, and for all $1\leq i<j\leq \chi$, we have that $S^g_{ij}<0$ (recall $S^g_{ij}$ from Equation \eqref{eq:S_ij}).
    \item For every eigenfunction $g$ with eigenvalue $\lambda_N$, we have that $$\RQ(g_{ij}) = \RQ(g_{12})=\lambda_1 = \frac{\chi-\lambda_N}{\chi-1}$$ for all $1\leq i<j\leq \chi$. Hence, the functions $g_{ij}$ from Equation \eqref{eq:g_ij} for $i\neq j$ are eigenfunctions with eigenvalue $\lambda_1$.
\end{enumerate}
Hence, the multiplicity of $\lambda_1 = (\chi-\lambda_N)/(\chi-1)$ equals at least $\chi-1$, and this inequality is strict if $m_\Gamma(\lambda_N)>1$.
\end{proposition}

\begin{proposition}\label{prop:lambdaN-1}
     Let $\Gamma$ be an oriented hypergraph such that $A(\Gamma)\neq \boldsymbol 0$ and  $$\chi=\frac{\lambda_N-\lambda_1}{\lambda_N-1}.$$ Then, the following hold:
    \begin{enumerate}
    \item For every eigenfunction $h$ with eigenvalue $\lambda_1$, for all proper $\chi$-colorings with corresponding coloring classes $V_1,\ldots,V_\chi$, we have
    \[
    \supp(h) \cap V_i \neq \varnothing \text{ for all } 1\leq i\leq \chi.
    \]
    \item For every eigenfunction $h$ with eigenvalue $\lambda_1$, and for all $1\leq i<j\leq \chi$, we have that $S^h_{ij}\geq 0$.
    \item For every eigenfunction $h$ with eigenvalue $\lambda_1$, we have that $$\RQ(h_{ij}) = \RQ(h_{12})=\lambda_N = \frac{\chi-\lambda_1}{\chi-1}$$ for all $1\leq i<j\leq \chi$. Hence, the functions $h_{ij}$ from Equation \eqref{eq:h_ij} for $i\neq j$ are eigenfunctions with eigenvalue $\lambda_N$.
\end{enumerate}
Hence, the multiplicity of $\lambda_N = (\chi-\lambda_1)/(\chi-1)$ equals at least $\chi-1$, and this inequality is strict if $m_\Gamma(\lambda_1)>1$.
\end{proposition}

As we mentioned earlier,  without additional information about the eigenspaces of $\lambda_1$ or $\lambda_N$, 
it is not possible to derive more specific properties of hypergraphs that attain equality in Theorem \ref{thm:cor5.4AMZ} beyond those already given in Propositions \ref{prop:1-lambda1} and \ref{prop:lambdaN-1}.

However, for $c$-uniform unoriented hypergraphs,
the eigenspace associated with $\lambda_N$ is known. This allows us to deduce explicit structural properties of such hypergraphs, specifically regarding how vertices from different coloring classes are distributed across the edges.

\subsection{Strong vertex coloring of unoriented \texorpdfstring{$c$}{c}-uniform hypergraphs}\label{sec:uniformhypergraphs}
For the rest of this section, we let $c\geq 2$ and we fix a $c$-uniform hypergraph $\Gamma$ in which all vertices are inputs in all edges they belong to. We begin by presenting a corollary of Proposition \ref{prop:1-lambda1}, followed by two examples of $c$-uniform hypergraphs for which the bound from Theorem \ref{thm:cor5.4AMZ} is sharp. We then state two results describing how vertices from different coloring classes are distributed across the edges, of which we shall prove a more general version in Section \ref{sec:d-proper}.\\

Note that, since we are assuming that $\Gamma$ is $c$-uniform, for every pair of distinct vertices $v,w\in V(\Gamma)$, we have that
\[
A_{v,w} = -\bigl| \{ e\in E \colon v,w\in e\}\bigr|, \text{ hence } L_{v,w} = \frac{\bigl| \{ e\in E \colon v,w\in e\}\bigr|}{\deg v}.
\]
Moreover, in this case, we know from Proposition \ref{prop:eigenvalues} that $\lambda_N = c$, and that this eigenvalue has multiplicity $1$, with eigenspace consisting of the constant functions.
Since $\lambda_1\geq 0$, it follows that
\[
\min\bigl\{\lambda_N-1,1-\lambda_1\bigr\} = 1-\lambda_1.
\]
Hence, the bound in Theorem \ref{thm:cor5.4AMZ} becomes
\[
\chi(\Gamma)\geq \frac{c-\lambda_1}{1-\lambda_1}.
\]
Furthermore, for $c$-uniform unoriented hypergraphs, the function defined in Equation \eqref{eq:g_ij} in the proof of Theorem \ref{thm:cor5.4AMZ} can be equivalently defined as follows.
\begin{definition}\label{def:gij}
    Let $V_1,\ldots,V_k$ be the coloring classes with respect to some fixed proper $k$-coloring and let $i,j\in \{1,\ldots,k\}$ be distinct. Then, the function $g_{ij}\colon V\to \R$ is defined by
    \[
    g_{ij}(w) = \begin{cases}
        1, &\text{ if }w\in V_i,\\
        -1, &\text{ if }w\in V_j,\\
        0, &\text{ otherwise}
    \end{cases}
    \]
\end{definition}
Proposition \ref{prop:1-lambda1} has the following corollary.

\begin{corollary}\label{cor:chi=c}

\begin{enumerate}
    \item  If $\chi=c$, then $\lambda_1=0$, and the eigenvalue $0$ has multiplicity at least $c-1$.
    \item If $\lambda_1=0$, then the bound from Theorem \ref{thm:cor5.4AMZ} is sharp if and only if $\chi=c$.
\end{enumerate}
\end{corollary}

\begin{proof} If $\chi=c$, then the bound from Theorem \ref{thm:cor5.4AMZ} can be rewritten as follows:
    \[
    \lambda_1\leq \frac{\chi-c}{\chi-1} = 0.
    \]
    Since $\lambda_1\geq0$, the bound must be sharp. By Proposition \ref{prop:1-lambda1}, this implies that the multiplicity of the eigenvalue $0$ is at least $\chi-1=c-1$.\\

  If $\lambda_1=0$, then the bound from Theorem \ref{thm:cor5.4AMZ} becomes $\chi \geq \lambda_N=c$, which is trivially true. Moreover, the bound is sharp if and only if $\chi=c$.

\end{proof}

\begin{example}\label{ex:hyperflowers}
Examples of hypergraphs with $\lambda_1=0$ and $\chi=c$ include $c$-uniform hyperflowers. A \emph{hyperflower} is a hypergraph $\Gamma=(V,E)$ such that there exists a non-empty proper subset $h_k\subsetneq V$ of \emph{central vertices} satisfying  $e\cap e' = h_k$ for all distinct $e,e'\in E$. The edges of a hyperflower are also called \emph{petals}. The spectrum of the \emph{$c$-uniform hyperflower $H_{p,k}^c$} with $p$ petals and $k$ central vertices is given by
\begin{equation}\label{eq:spectrumhyperflower}
\biggl\{ c^{(1)}, \bigl(c-k\bigr)^{(p-1)}, 0^{(N-p)} \biggr\},
\end{equation}
as shown in Lemma 6.12 of \cite{AndreottiMulas2022}.
Two examples of $3$-uniform hyperflowers are shown in Figure \ref{fig:petalgraphs}.
\begin{figure}[h]
    \centering
    \includegraphics[width=0.6\linewidth]{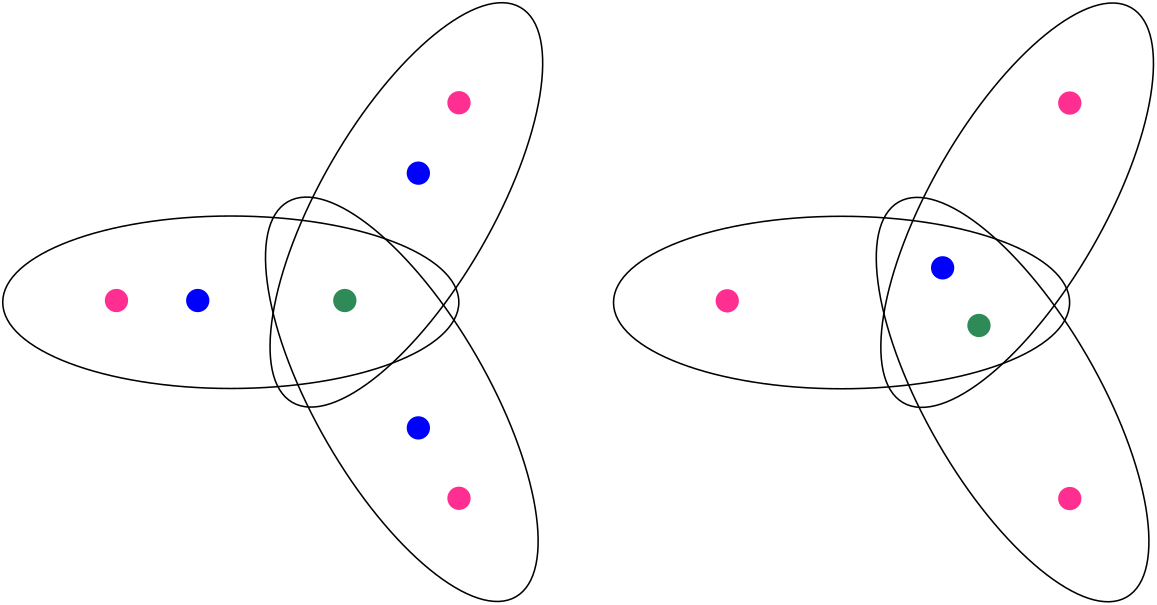}
    \caption{The two $3$-uniform hyperflowers $H_{3,1}^3$ (left) and $H_{3,2}^3$ (right) with $3$ petals and with $1$ and $2$ central vertices, respectively.}
    \label{fig:petalgraphs}
\end{figure}
\end{example}

An example of hypergraphs with $\lambda_1>0$ for which the bound from Theorem \ref{thm:cor5.4AMZ} is sharp is given by complete multipartite hypergraphs with equal-sized partition classes. These generalize complete multipartite graphs with equal-sized partition classes, which are a special case of Turán graphs. We now provide their definition and describe their spectrum. 
\begin{definition}\label{def:k-partite}
    Let $k\geq c$ and $s_i>0$, for $i\in \{1,\ldots,k\}$. The \emph{complete $k$-partite $c$-uniform hypergraph} with partition classes of sizes $s_1,\ldots,s_k$ is the hypergraph $\Gamma^c_{s_1,\ldots,s_k}$ whose vertex set $V$ can be partitioned into $k$ disjoint sets $V_1,\ldots,V_k$ of sizes $s_i = |V_i|$, such that \[ e\subseteq V \text{ is an edge of } \Gamma^c_{s_1,\ldots,s_k} \iff |e| = c \text{ and } |e\cap V_i|\leq 1 \text{ for all } i\in \{1,\ldots,k\}.\]

    We also let $\Gamma^c(s,k)\coloneqq \Gamma^c_{\underbrace{s,\ldots,s}_k}$.

\end{definition}
Hence, $\Gamma^c_{s_1,\ldots,s_k}$ contains all possible edges $e$ of cardinality $c$ such that no two vertices in $e$ belong to the same partition class $V_i$.

\begin{proposition}\label{prop:spectrumcompmult}
    Let $k\geq c$ and let $s\geq1$. The spectrum of $\Gamma^c(s,k)$ equals
    \[
    \biggl\{ c^{(1)},1^{k(s-1)},\frac{k-c}{k-1}^{(k-1)}\biggl\}.
    \]
\end{proposition}
\begin{proof}
  We let $$V_i \coloneqq \bigl\{v_1^i,\ldots,v_s^i\bigr\}$$ be the partition classes of $\Gamma^c(s,k)$.
    Recall that, by Remark \ref{rmk:lambdaN=c}, since $\Gamma^c(s,k)$ is connected, it has eigenvalue $c$ with multiplicity $1$.
    Furthermore, for $1\leq i\leq k$ and $2\leq j\leq s$, let
  \[
    f_{v_1^i,v_j^i}(v)\coloneqq \begin{cases}
        1, &\text{ if } v = v_1^i,\\
        -1, &\text{ if } v = v_j^i,\\
        0, &\text{ otherwise.}
    \end{cases}
    \]
    These are $k(s-1)$ linearly independent eigenfunctions with eigenvalue $1$.
    
    Lastly, for $1\leq i<j\leq k$, the functions $g_{ij}$ from Definition \ref{def:gij} are $k-1$ linearly independent eigenfunctions with eigenvalue $\frac{k-c}{k-1}$.
\end{proof}
Now, observe that $\chi\biggl(\Gamma^c(s,k)\biggr) = k$. This leads to the following corollary.
\begin{corollary}
   For $\Gamma^c(s,k)$, the bound in Theorem \ref{thm:cor5.4AMZ} is sharp.
\end{corollary}


Further results on strong colorings, derived from the broader context of $d$-proper colorings, will be given in Theorem \ref{thm:uniformstrong} and Corollary \ref{cor:equality} at the end of Section \ref{sec:d-proper}.


\section{Spectral bounds for a generalized notion of hypergraph coloring}\label{sec:d-proper}
\subsection{\texorpdfstring{$d$}{d}-proper colorings of \texorpdfstring{$c$}{c}-uniform unoriented hypergraphs}

In this section, we study the notion of $d$-proper vertex colorings, as introduced in Definition \ref{def:d-proper1}. These generalize strong colorings (Definition \ref{def:coloring1}) by allowing up to $d$ vertices of the same color per edge.


\begin{example}\label{ex:d-proper}
    Consider a $c$-uniform hypergraph in which  the vertices represent the employees of a given company, while the edges represent project teams. Suppose that at most $d$ employees can be absent from a project team without overburdening the remaining members. Then, a $d$-proper coloring corresponds to a partition of the employees into groups that can have vacation at the same time.
\end{example}

Also in this section, we let $c\geq 2$ and we fix a $c$-uniform hypergraph $\Gamma$ in which all vertices are inputs in all edges they belong to. Recall from Definition \ref{def:d-proper1} that the $d$-proper coloring number of $\Gamma$ is denoted by $\chi_d(\Gamma)=\chi_d$, and it is the smallest positive integer such that there exists a $d$-proper $\chi_d$-coloring of the vertices of $\Gamma$.

\begin{remark}
    Since $\Gamma$ is $c$-uniform, then we
    have the following trivial bound:
    \[
    \chi_d \geq \frac cd.
    \]
\end{remark}
We now state and prove the main theorem of this section.

\begin{theorem}\label{thm:d-proper}
    Let $d$ be an integer such that $1\leq d\leq c-1$. Then,
    \[
    \chi_d \geq \frac{c-\lambda_1}{d-\lambda_1}.
    \]
    Moreover, the bound is sharp if and only if $d\mid c$.
\end{theorem}
\begin{proof}
    Recall that $\lambda_N=c$ and that the constant functions are corresponding eigenfunctions. Let $V_1,\ldots,V_{\chi_d}$ be the coloring classes with respect to a fixed $d$-proper $\chi_d$-coloring. Then,
    \begin{align*}
      \nonumber  c =& \RQ(\boldsymbol 1) = 1 - 2\frac{\sum_{v\neq w}A_{v,w}}{\sum_{v\in V}\deg v}\\
   \nonumber     =& 1 + 2\frac{\sum_{v\neq w}\bigl| A_{v,w}\bigr|}{\sum_{v\in V}\deg v}\\
    \nonumber    =& 1 + 2\frac{\sum_{i\neq j}\sum_{\substack{v\in V_i\\ w\in V_j}}\bigl| A_{v,w}\bigr|}{\sum_{v\in V}\deg v}+2\frac{\sum_i\sum_{\substack{v,w\in V_i\\ v\neq w}}\bigl| A_{v,w}\bigr|}{\sum_{v\in V}\deg v}.
    \end{align*}
    By rewriting the denominator of the first fraction and substituting the coefficient $2$ in front of the second fraction by $2\chi_d - 2(\chi_d - 1)$, we obtain 
    \begin{align}
     \nonumber c   =& 1 + 2\frac{\sum_{i\neq j}\sum_{\substack{v\in V_i\\ w\in V_j}}\bigl| A_{v,w}\bigr|}{\frac1{\chi_d-1}\sum_{i\neq j}\sum_{v\in V_i\cup V_j}\deg v} -2(\chi_d-1)\frac{\sum_i\sum_{\substack{v,w\in V_i\\ v\neq w}}\bigl| A_{v,w}\bigr|}{\sum_{v\in V}\deg v}\\
        &+2\chi_d\frac{\sum_i\sum_{\substack{v,w\in V_i\\ v\neq w}}\bigl| A_{v,w}\bigr|}{\sum_{v\in V}\deg v}.  \label{eq:fractions}
    \end{align}    
    We shall bound the last fraction in \eqref{eq:fractions} later. We first focus on rewriting and bounding the difference between the first two fractions. We have
    \begin{align*}
        & 2\frac{\sum_{i\neq j}\sum_{\substack{v\in V_i\\ w\in V_j}}\bigl| A_{v,w}\bigr|}{\frac1{\chi_d-1}\sum_{i\neq j}\sum_{v\in V_i\cup V_j}\deg v} -2(\chi_d-1)\frac{\sum_i\sum_{\substack{v,w\in V_i\\ v\neq w}}\bigl| A_{v,w}\bigr|}{\sum_{v\in V}\deg v} \\
        =& 2\bigl(\chi_d-1\bigr)\frac{\sum_{i\neq j}\sum_{\substack{v\in V_i\\ w\in V_j}}\bigl| A_{v,w}\bigr|}{\sum_{i\neq j}\sum_{v\in V_i\cup V_j}\deg v} -2(\chi_d-1)\frac{\bigl(\chi_d-1\bigr)\sum_i\sum_{\substack{v,w\in V_i\\ v\neq w}}\bigl| A_{v,w}\bigr|}{(\chi_d-1)\sum_{v\in V}\deg v} \\
        =& 2\bigl(\chi_d-1\bigr)\frac{\sum_{i\neq j}\sum_{\substack{v\in V_i\\ w\in V_j}}\bigl| A_{v,w}\bigr|}{\sum_{i\neq j}\sum_{v\in V_i\cup V_j}\deg v} -2(\chi_d-1)\frac{\sum_{i\neq j}\sum_{\substack{v,w\in V_i \text{ or }\\ v,w\in V_j}}\bigl| A_{v,w}\bigr|}{\sum_{i\neq j}\sum_{v\in V_i\cup V_j}\deg v}\\
        =& 2(\chi_d-1) \frac{\sum_{i\neq j}\biggl(\sum_{\substack{v\in V_i\\w\in V_j}}\bigl| A_{v,w}\bigr| - \sum_{\substack{v,w\in V_i \text{ or }\\ v,w\in V_j}}\bigl| A_{v,w}\bigr|\biggr)}{\sum_{i\neq j}\sum_{v\in V_i\cup V_j}\deg v}\\
        \leq& 2(\chi_d-1) \frac{\sum_{i\neq j}\max\biggl\{0,\sum_{\substack{v\in V_i\\w\in V_j}}\bigl| A_{v,w}\bigr| \sum_{\substack{v,w\in V_i \text{ or }\\ v,w\in V_j}}\bigl| A_{v,w}\bigr|\biggr\}}{\sum_{i\neq j}\sum_{v\in V_i\cup V_j}\deg v}.
    \end{align*}
    To  further bound the last expression, we consider two cases. \\

    \textbf{Case $1$:} 
    \[
    \sum_{i\neq j}\max\biggl\{0,\sum_{\substack{v\in V_i\\w\in V_j}}\bigl| A_{v,w}\bigr| - \sum_{\substack{v,w\in V_i \text{ or }\\ v,w\in V_j}}\bigl| A_{v,w}\bigr|\biggr\}>0.
    \]
    In this case, there exist $i\neq j$ such that $$\sum_{\substack{v\in V_i\\w\in V_j}}\bigl| A_{v,w}\bigr| - \sum_{\substack{v,w\in V_i \text{ or }\\ v,w\in V_j}}\bigl| A_{v,w}\bigr|>0,$$ and
    we may assume without loss of generality that
    \[
    \max_{i\neq j}\biggl\{\frac{\sum_{\substack{v\in V_i\\w\in V_j}}\bigl| A_{v,w}\bigr| - \sum_{\substack{v,w\in V_i \text{ or }\\ v,w\in V_j}}\bigl| A_{v,w}\bigr|}{\sum_{v\in V_i\cup V_j}\deg v}\biggr\}
    \]
    is attained by $i=1$ and $j=2$.
    Using the function $g_{12}$ as in Definition \ref{def:gij}, we have
    \begin{align*}
        &2(\chi_d-1) \frac{\sum_{i\neq j}\max\left\{0,\sum_{\substack{v\in V_i\\w\in V_j}}\bigl| A_{v,w}\bigr| - \sum_{\substack{v,w\in V_i \text{ or }\\ v,w\in V_j}}\bigl| A_{v,w}\bigr|\right\}}{\sum_{i\neq j}\sum_{v\in V_i\cup V_j}\deg v}\\
        \leq& 2\bigl(\chi_d-1\bigr) \frac{\sum_{\substack{v\in V_1\\w\in V_2}}\bigl| A_{v,w}\bigr| - \sum_{\substack{v,w\in V_1 \text{ or }\\ v,w\in V_2}}\bigl| A_{v,w}\bigr|}{\sum_{\substack{v\in V_1\cup V_2}}\deg v}\\
        \overset{(i)}{=}& \bigl(\chi_d-1\bigr)\bigl(1-\RQ(g_{12})\bigr) \\
        {\leq}& \bigl(\chi_d-1\bigr)\bigl(1- \lambda_{1}\bigr),
    \end{align*}
    where, in $(i)$, we used the fact that
    \begin{align*}
        \RQ\bigl(g_{12}\bigr) &= 1 + 2\frac{\sum_{v\neq w}\bigl|A_{v,w}\bigr| \cdot g_{12}(v) \cdot g_{12}(w)}{\sum_{v\in V}\deg v\cdot  g_{12}(v)^2}\\
        &= 1+ 2 \frac{\sum_{\substack{v\in V_1\\w\in V_2}}\bigl|A_{v,w}\bigr| \cdot g_{12}(v) \cdot g_{12}(w) +\sum_{\substack{v,w\in V_1 \text{ or } \\ v,w\in V_2}} \bigl|A_{v,w}\bigr| \cdot g_{12}(v) \cdot g_{12}(w)}{\sum_{v\in V_1 \cup V_2}\deg v\cdot g_{12}(v)^2}\\
        &= 1+2\frac{-\sum_{\substack{v\in V_1\\w\in V_2}}\bigl|A_{v,w}\bigr| +\sum_{\substack{v,w\in V_1 \text{ or } \\ v,w\in V_2}} \bigl|A_{v,w}\bigr|}{\sum_{v\in V_1 \cup V_2}\deg v}.
    \end{align*}

\textbf{Case $2$:}
    \[
    \sum_{i\neq j}\max\biggl\{0,\sum_{\substack{v\in V_i\\w\in V_j}}\bigl| A_{v,w}\bigr| - \sum_{\substack{v,w\in V_i\\ v\neq w}}\bigl|A_{v,w}\bigr| - \sum_{\substack{v,w\in V_j\\v\neq w}}\bigl|A_{v,w}\bigr|\biggr\}=0.
    \]
    Then,
    \begin{align*}
        &2(\chi_d-1) \frac{\sum_{i\neq j}\max\biggl\{0,\sum_{\substack{v\in V_i\\w\in V_j}}\bigl| A_{v,w}\bigr| - \sum_{\substack{v,w\in V_i\\ v\neq w}}\bigl|A_{v,w}\bigr| - \sum_{\substack{v,w\in V_j\\v\neq w}}\bigl|A_{v,w}\bigr|\biggr\}}{\sum_{i\neq j}\sum_{v\in V_i\cup V_j}\deg v} = 0\\
        <& \bigl(\chi_d-1\bigr)\bigl(1-\lambda_1\bigr),
    \end{align*}
    since $\lambda_1<1$ and $\sum_{i\neq j}\sum_{v\in V_i\cup V_j}\deg v>0$.\newline

   Hence, Cases 1 and 2 show that the difference between the first two fractions in \eqref{eq:fractions} is bounded above by $\bigl(\chi_d-1\bigr)\bigl(1-\lambda_1\bigr)$.\newline 

    We now bound the last  fraction appearing in \eqref{eq:fractions}. We have that

    \begin{align*}
        2\chi_d\frac{\sum_i\sum_{\substack{v,w\in V_i\\ v\neq w}}\bigl| A_{v,w}\bigr|}{\sum_{v\in V}\deg v}
        &= 2\chi_d\frac{\frac12 \sum_{v\in V}\sum_{\substack{w\in V\colon w \text{ in same}\\ \text{coloring class as $v$}}}\bigl|A_{v,w}\bigr|}{\sum_{v\in V}\deg v}\\
        &\overset{(ii)}{\leq} \chi_d \frac{\sum_{v\in V}\deg v(d-1)}{\sum_{v\in V}\deg v}\\
        &=(d-1)\chi_d,
    \end{align*}
    where, in $(ii)$, we used the fact that, for every vertex $v\in V$ and every edge $e\ni v$, there are at most $d-1$ other vertices in $e$ that have the same color as $v$.\newline
    
    Combining \eqref{eq:fractions} with the bounds obtained above, we conclude that
    \begin{align*}
        c 
        &\leq 1 + \bigl(\chi_d-1\bigr)\bigl(1- \lambda_{1}\bigr) + \chi_d(d-1),
    \end{align*}
    which is equivalent to
    \[
    \chi_d \geq \frac{c-\lambda_1}{d-\lambda_1}.
    \]    
    Finally, we need to show that the bound is sharp if and only if $d\mid c$. \\
     Note that, if the inequality in $(ii)$
    is an equality, then, for every coloring with corresponding coloring classes $V_1,\ldots,V_{\chi_d}$, for every edge $e\in E$, and for every coloring class $V_i$, we must have $|e\cap V_i| \in \bigl\{0,d\bigr\}$. Hence, $d\mid c$.

       
    Now we want to show that, if $d\mid c$, then the bound in the statement is sharp, that is, there exists a $c$-uniform hypergraph for which $$\chi_d = \frac{c-\lambda_1}{d-\lambda_1}.$$

    Consider a $c$-uniform hyperflower. Then, we know that $\lambda_1=0$, and therefore the bound becomes $\chi_d\geq c/d$. To show that this is an equality,  observe that one can always color the vertices of a $c$-uniform hyperflower using exactly $c/d$ colors, such that each edge contains exactly $d$ vertices of every color. See \cref{fig:hyperflowerbound} for the case where $c=9$ and $d=3$.  Therefore, $\chi_d= c/d$.

\end{proof}

 \begin{figure}
        \centering
        \includegraphics[width=0.4\linewidth]{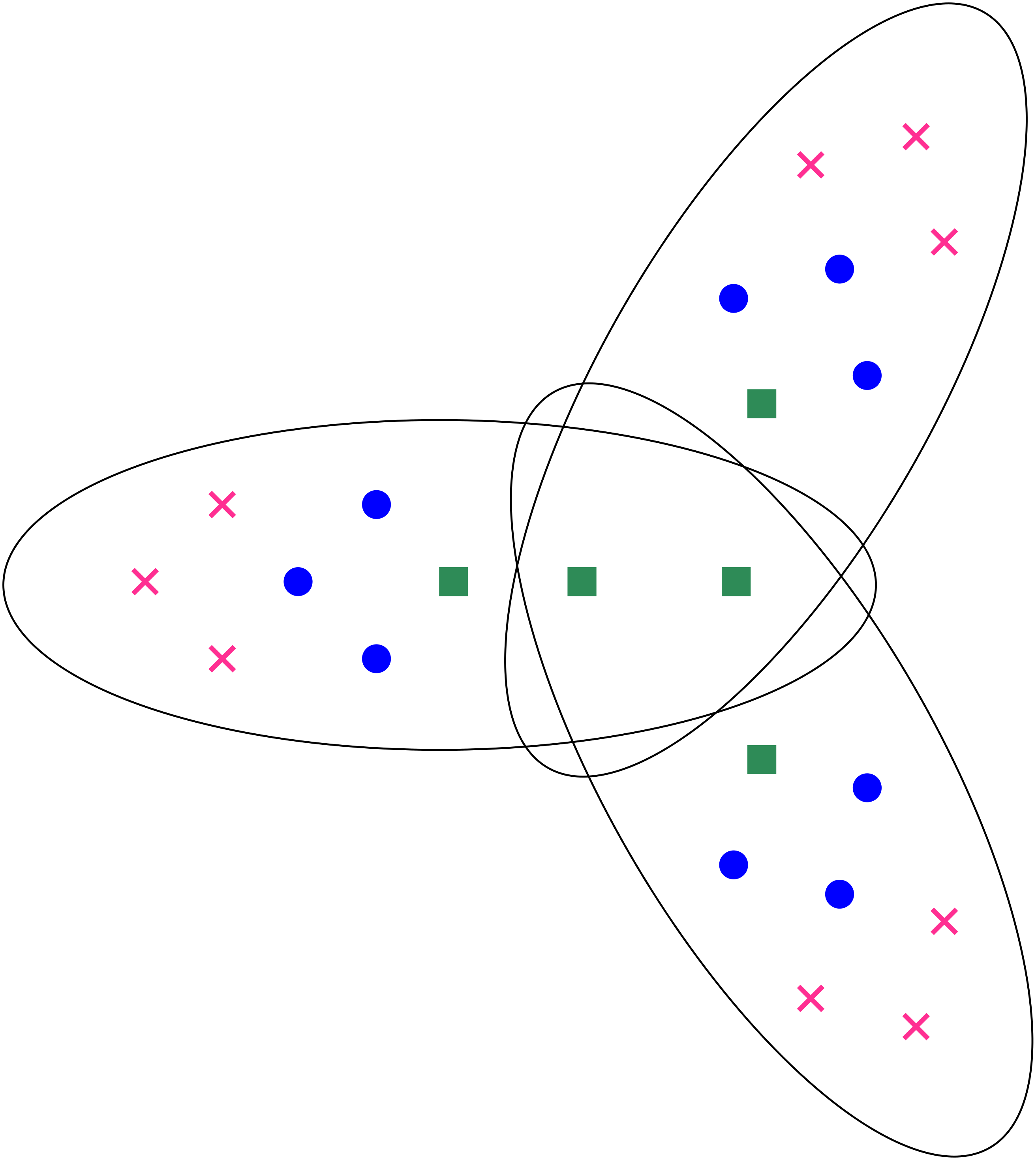}
        \caption{The $9$-uniform hyperflower $H_{3,2}^9$, for which the bound in \cref{thm:d-proper} is sharp.}
        \label{fig:hyperflowerbound}
    \end{figure}



In Corollary \ref{cor:d-propersharp} below, we provide necessary conditions for a hypergraph to attain equality in the bound from Theorem \ref{thm:d-proper}. To that end, we first prove the following theorem.

\begin{theorem}\label{thm:uniformd-proper}
     Let $\Gamma$ be a $c$-uniform oriented hypergraph for which all vertices are inputs in all edges they belong to. Moreover, let $V_1,\ldots,V_k$ be a partition of the vertex set, and assume that there exists $m\geq 1$ such that, for every edge $e\in E$ and for all $i=1,\ldots,k$, 
     $$\bigl| e\cap V_i\bigr|\in \bigl\{0,m\bigr\}.$$ Then, the functions $g_{ij}$ from Definition \ref{def:gij} are eigenfunctions if and only if, for all $i=1,\ldots,k$ and for all $v\in V$, we have
    \begin{equation}\label{eq:assumption}
    \bigl| \{e\in E \colon v\in e \text{ and } e\cap V_i \neq \varnothing\}\bigr| = \begin{cases}
        \frac{(c/m-1) \deg v}{k -1}, &\text{ if } v\notin V_i,\\
        \deg v, &\text{ if } v\in V_i.
    \end{cases}
    \end{equation}
   In this case, the corresponding eigenvalue is $(m\cdot k-c)/(k-1)$.
\end{theorem}
\begin{proof}
    For the ``if'' direction, let $1\leq i\neq j\leq k$, and consider vertices $v_i\in V_i$,  $v_j\in V_j$. Let also $v\in V\setminus \bigl(V_i\cup V_j\bigr)$. Then, we have that
    \begin{align*}
        Lg_{ij}(v) &= \sum_{w\in V}L_{v,w}g_{ij}(w)\\
        &= \sum_{w\in V_i}\frac{|A_{v,w}|}{\deg v} - \sum_{w\in V_j}\frac{|A_{v,w}|}{\deg v}\\
        &= \frac1{\deg v} \biggl( \bigl| \bigl\{\bigl(e,w_i\bigr)\colon v\in e, w_i \in e\cap V_i\bigr\}\bigr| - \bigl| \bigl\{\bigl(e,w_j\bigr)\colon v\in e, w_j \in e\cap V_j\bigr\}\bigr| \biggr)\\
        &= \frac m{\deg v}\biggl(\frac{(c/m-1) \deg v}{k -1} - \frac{(c/m-1) \deg v}{k -1}\biggr)\\
        &= 0,
           \end{align*} while
             \begin{align*}
        Lg_{ij}(v_i) &= \sum_{w\in V}L_{v_i,w}g_{ij}(w)\\
        &= 1+\sum_{w\in V_i}\frac{|A_{v_i,w}|}{\deg v_i}-\sum_{w\in V_j}\frac{|A_{v_i,w}|}{\deg v_i}\\
        &= \frac 1{\deg v_i}\biggl(\bigl|\bigl\{(e,w_i)\colon v_i\in e, w_i\in e\cap V_i\bigr\}\bigr|-\bigl|\bigl\{(e,w_i)\colon v_i\in e, w_i\in e\cap V_j\bigr\}\bigr| \biggr),
    \end{align*}
    where the last equality follows from the definition of the adjacency matrix and from the observation that $v_i$ appears in exactly $\deg v_i$ ordered pairs $(e,v_i)$, ranging over all edges $e$ such that $v_i \in e$.

    Next, recall our assumption that for every edge $e \in E$ and for each index $i = 1,\ldots,k$, one has $\lvert e \cap V_i \rvert \in \{0,m\}$. In particular, if $v_i \in e \cap V_i$, then the intersection $e \cap V_i$ must contain exactly $m$ vertices, which implies the existence of precisely $m-1$ additional vertices $w \in e \cap V_i$ distinct from $v_i$. Combining this observation with \eqref{eq:assumption}, we obtain
    \begin{align*}
        Lg_{ij}(v_i) &= \frac 1{\deg v_i}\bigl(\deg v_i\cdot m - m\cdot \bigl|\bigl\{e\colon v_i\in e, e\cap V_j\neq \varnothing\bigr\}\bigr|\bigr)\\
        &=m - \frac{m}{\deg v_i}\cdot \frac{(c/m-1)\deg v_i}{k-1}\\
        &= \frac{m(k-1)}{k-1}-\frac{c-m}{k-1}\\
        &=\frac{m\cdot k-c}{k-1}.
    \end{align*}
    Analogously, one can show that
    \begin{align*}
        Lg_{ij}(v_j) &= \frac{c-m\cdot k}{k-1}.
    \end{align*}
    Hence, $g_{ij}$ is an eigenfunction with eigenvalue $(m\cdot k-c)/(k-1)$.\newline 

    To prove the ``only if'' direction, assume that the functions $g_{ij}$ are eigenfunctions. Then, let $v\in V$, and fix distinct $i,j\in \{1,\ldots,k\}$ such that $v\notin V_i\cup V_j$. Then,
    \begin{align*}
        0 &= Lg_{ij}(v)\\
        &= \frac1{\deg v} \biggl( \bigl| \bigl\{\bigl(e,w_i\bigr)\colon v\in e, w_i \in e\cap V_i\bigr\}\bigr| - \bigl| \bigl\{\bigl(e,w_j\bigr)\colon v\in e, w_j \in e\cap V_j\bigr\}\bigr| \biggr).
    \end{align*}
    Hence, for any $i,j$ such that $v\notin V_i\cup V_j$, we have that
    \[
    \bigl| \bigl\{\bigl(e,w_i\bigr)\colon v\in e, w_i \in e\cap V_i\bigr\}\bigr| = \bigl| \bigl\{\bigl(e,w_j\bigr)\colon v\in e, w_j \in e\cap V_j\bigr\}\bigr|.
    \]
    Since
    \[
    \sum_{\substack{l=1,\ldots,k\\ v\notin V_l}}\bigl| \bigl\{\bigl(e,w_l\bigr)\colon v\in e, w_l \in e\cap V_l\bigr\}\bigr| = \deg v \bigl(c-m\bigr)
    \]
    and since $|e\cap V_i|\in \{0,m\}$ for all $e\in E$ and $i=1,\ldots,k$,
    we conclude that
    \[
    \bigl| \{e\in E \colon v\in e \text{ and } e\cap V_i \neq \varnothing\}\bigr| = \begin{cases}
        \frac{(c/m-1) \deg v}{k -1}, &\text{ if } v\notin V_i,\\
        \deg v, &\text{ if } v\in V_i. 
    \end{cases}\qedhere
    \]
\end{proof}
Theorem \ref{thm:uniformd-proper} has the following corollary, which provides necessary conditions for a hypergraph to attain equality in the bound from Theorem \ref{thm:d-proper}.

\begin{corollary}\label{cor:d-propersharp}
   If Theorem \ref{thm:d-proper} holds with equality, then for any $d$-proper $\chi_d$-coloring with coloring classes $V_1,\ldots,V_{\chi_d}$, the following hold:
    \begin{enumerate}
        \item For each edge $e\in E$ and for all $i=1,\ldots,\chi_d$,
        \[
        \bigl|e\cap V_i\bigr| \in \bigl\{0,d \bigr\}.
        \]
        \item For all distinct $i,j \in\{ 1,\ldots,\chi_d\}$, the functions $g_{ij}$ from Definition \ref{def:gij} are eigenfunctions with eigenvalue \[\lambda_1 = \frac{d\cdot \chi_d - c}{\chi_d-1}.\]
        \item For all distinct $i,j\in\{1,\ldots,\chi_d\}$, and for all $v_i\in V_i$, we have
        \[
        \bigl| \bigl\{ e\colon v_i\in e \text{ and } e\cap V_j\neq \varnothing \bigr\}\bigr| = \frac{\deg v_i\bigl( c/d-1\bigr)}{\chi_d-1}.
        \]
    \end{enumerate}

    Moreover, even if the bound in Theorem \ref{thm:d-proper} is not attained, statement \textup{(1)} implies that statements \textup{(2)} and \textup{(3)} are equivalent.
\end{corollary}
\begin{proof}
    Assume that Theorem \ref{thm:d-proper} holds with equality. Then, we know by the proof of Theorem \ref{thm:d-proper} that, for all $e\in E$ and $i=1,\ldots,\chi_d$, we have $\bigl|e\cap V_i\bigr|\in \bigl\{0,d\bigr\}$, which proves (1).
    
    Furthermore, for all distinct $i,j\in\{1,\ldots,\chi_d\}$, we have that $\RQ\bigl(g_{ij}\bigr)=\lambda_1$. Therefore, the functions $g_{ij}$ are eigenfunctions with eigenvalue $\lambda_1$. By applying Theorem \ref{thm:uniformd-proper} with $m=d$ and $k=\chi_d$, we conclude that (2) and (3) are both true.
    
    Finally, even if the bound in Theorem \ref{thm:d-proper} is not attained, then by Theorem \ref{thm:uniformd-proper} we infer that, if condition (1) is satisfied, then conditions (2) and (3) are equivalent.
\end{proof}
We now state Theorem \ref{thm:uniformd-proper} and Corollary \ref{cor:d-propersharp} in the special case where $m=1$, so that the partition $V_1,\ldots,V_k$ corresponds to a strong vertex coloring (Definition \ref{def:coloring1}).
\begin{theorem}\label{thm:uniformstrong}
    Let $\Gamma$ be a $c$-uniform oriented hypergraph with only positive signs, and let $V_1,\ldots,V_k$ be the coloring classes of a fixed proper strong $k$-coloring. Then, the functions $g_{ij}$ from Definition \ref{def:gij} are eigenfunctions if and only if
    \[
    \bigl| \{e\in E \colon v\in e \text{ and } e\cap V_i \neq \varnothing\}\bigr| = \begin{cases}
        \frac{(c-1) \deg v}{k -1}, &\text{ if } v\notin V_i,\\
        \deg v, &\text{ if } v\in V_i.
    \end{cases}
    \]
    In this case, the corresponding eigenvalue is $(k-c)/(k-1)$.
\end{theorem}
The following corollary is a special case of Corollary \ref{cor:d-propersharp} and generalizes Theorem 7 from \cite{Gabriel-colouring}. 
\begin{corollary}\label{cor:equality}
     Let $\Gamma$ be a $c$-uniform oriented hypergraph with only positive signs, and let $V_1,\ldots,V_\chi$ be the coloring classes of a fixed proper strong $\chi$-coloring. If Theorem \ref{thm:cor5.4AMZ} holds with equality, then for all $v\in V$ and for all $1\leq i\leq \chi$, we have that
    \[
    \bigl| \{e\in E \colon v\in e \text{ and } e\cap V_i \neq \varnothing\}\bigr| = \begin{cases}
        \frac{(c-1) \deg v}{\chi -1}, &\text{ if } v\notin V_i,\\
        \deg v, &\text{ if } v\in V_i.
    \end{cases}
    \]
\end{corollary}
Corollary \ref{cor:equality} shows that, if Theorem \ref{thm:cor5.4AMZ} holds with equality (for $c$-uniform oriented hypergraphs with only positive signs), then the neighbors of any vertex $v$ are evenly distributed among the coloring classes of any proper strong $\chi$-coloring.

\subsection{\texorpdfstring{$q$}{q}-tailored colorings}\label{subsec:q-tailored}

We shall now focus on $q$-tailored colorings (Definition \ref{def:q-tailored1}), which generalize $d$-proper colorings in the context of $c$-uniform hypergraphs for which all vertices are inputs in all their edges. We begin by recalling the definition and the idea behind it from Section \ref{section:overview-colorings}.\newline

   \emph{Recall.} Let $\Gamma$ be a $c$-uniform hypergraph in which all vertices are inputs in all edges they belong to, and let $q\in [0,c-1]$. A (vertex) $k$-coloring $c\colon V(\Gamma)\to \{1,\ldots,k\}$ that has coloring classes $V_1,\ldots,V_k$ is \emph{$q$-tailored} if, for all $i=1,\ldots,k$ and for all $v\in V_i$, we have
    \begin{equation*}
        \sum_{w\in V_i}|A_{v,w}|\leq q \deg v.
    \end{equation*}
    The corresponding coloring number is denoted by $\chi^{q\text{-t}}\coloneqq \chi^{q\text{-t}}(\Gamma)$. It is the smallest positive integer such that there exists a $q$-tailored $\chi^{q\text{-t}}$-coloring of the vertices of $\Gamma$.\newline

In particular, if $\Gamma$ is a $c$-uniform hypergraph in which every vertex is an input in all edges it belongs to, then any $d$-proper coloring is also a $(d-1)$-tailored coloring. This is why $q$-tailored colorings generalize $d$-proper colorings in this setting.\\

Informally, a $q$-tailored coloring ensures that no vertex is too connected to others of the same color: the average number of same-colored neighbors of any vertex $v\in V$ in an edge is at most a fraction $q$ of its degree.


\begin{example}\label{ex:q-tailored}
    Consider a hypergraph in which the vertices represent professional tailors, and the edges represent collaborative projects, such as working on the same clothing line. We assume that every project requires the same amount of time from each tailor.
    
    We aim to partition the tailors into vacation groups such that the workload for the non-vacationing tailors does not become too high. To achieve this, we fix a fraction $q\in [0,c-1]$. A $q$-tailored coloring ensures that, for each tailor, the cumulative number of shared projects with vacationing colleagues, which is equal to the sum of absolute values of the adjacency entries of vertices belonging to the same vacation group,
    does not exceed this fraction $q$ of their total workload.
\end{example}

We generalize Theorem \ref{thm:d-proper} to the setting of $q$-tailored colorings, as follows.

\begin{theorem}\label{thm:q-tailored}
    Let $q\in [0,c-1]$. Then,
    \[
    \chi^{q\text{-t}} \geq \frac{c-\lambda_1}{q+1-\lambda_1}.
    \]
    Furthermore, the bound is sharp.
\end{theorem}
\begin{proof}
    The proof is analogous to the proof of Theorem \ref{thm:d-proper} (substituting $\chi_d$ by $\chi^{q\text{-t}}$ ), except for the bound of the last fraction in \eqref{eq:fractions}. In the more general definition of $q$-tailored colorings, we have:
    \begin{align*}
        2\chi^{q\text{-t}}\frac{\sum_i\sum_{\substack{v,w\in V_i\\ v\neq w}}\bigl| A_{v,w}\bigr|}{\sum_{v\in V}\deg v}
        &= 2\chi^{q\text{-t}}\frac{\frac12 \sum_{v\in V}\sum_{\substack{w\in V\colon w \text{ in same}\\ \text{coloring class as $v$}}}\bigl|A_{v,w}\bigr|}{\sum_{v\in V}\deg v}\\
        &\overset{(i)}{\leq} \chi^{q\text{-t}} \frac{\sum_{v\in V}q\cdot \deg v}{\sum_{v\in V}\deg v}\\
        &=q\chi^{q\text{-t}},
    \end{align*}
    where, in $(i)$, we applied the definition of $q$-tailored coloring.
    
    Combining the two bounds, we obtain
    \begin{align*}
        c 
        &\leq 1 + \bigl(\chi^{q\text{-t}}-1\bigr)\bigl(1- \lambda_{1}\bigr) + q\cdot\chi^{q\text{-t}},
    \end{align*}
    which is equivalent to
    \[
    \chi^{q\text{-t}} \geq \frac{c-\lambda_1}{q+1-\lambda_1}.
    \]
    Finally, by Theorem \ref{thm:d-proper}, we know that the bound is sharp.
\end{proof}
The following corollary follows by the proof of Theorem \ref{thm:q-tailored}.
\begin{corollary}\label{cor:q-tailored}
  If the bound in Theorem \ref{thm:q-tailored} holds as an equality, then for all $v\in V$,    
    \[
    \sum_{w\in V_i}|A_{v,w}| = q \deg v.
    \]
    Moreover, for every $q$-tailored coloring with $\chi^{q\text{-t}}$ colors and coloring classes $V_1,\ldots,V_{\chi^{q\text{-t}}}$, the functions $g_{ij}$ from Definition \ref{def:gij} are eigenfunctions with eigenvalue
    \[
    \lambda_1 = \frac{(q+1)\chi^{q\text{-t}}-c}{\chi^{q\text{-t}}-1}.
    \]
\end{corollary}

\section{Coloring graphs \texorpdfstring{$d$}{d}-improperly}\label{sec:d-improper}

In this section, we consider graphs instead of general hypergraphs.
Let $G=(V,E)$ be a graph with $|V|=N$. We let $L=I-D^{-1}A$ be its normalized Laplacian, and we write the eigenvalues of $L$ by
\[
0=\lambda_1\leq \cdots \leq \lambda_N.
\]

While the bound in Theorem \ref{thm:cor5.4AMZ} was proved in \cite{Abiad20} for hypergraphs, its graph version, $\chi\geq \lambda_N/\bigl(\lambda_N-1\bigr)$, was first established in \cite{Gabriel-colouring}. In this section, we prove a generalization of this inequality that involves the $d$-improper coloring number (Definition \ref{def:d-improper1}).\newline

\emph{Recall.} A vertex coloring $c\colon V\to \{1,\ldots,k\}$ is \emph{$d$-improper} if every vertex $v$ has at most $d$ neighbors with the same color as $v$. The \emph{$d$-improper coloring number}, denoted $\chi^d=\chi^d(G)$, is the minimal number of colors needed to color $G$ $d$-improperly.\newline 

\begin{remark}
   A proper vertex coloring of a graph is $0$-improper. Moreover, if $G$ is a $k$-regular graph and if $d\geq0$, then a $d/k$-tailored coloring (Definition \ref{def:q-tailored1}) corresponds to a $d$-improper coloring.
\end{remark}
Bilu (2006) \cite{Bilu2006} proved the following generalization of the Hoffman bound \cite{Hoffman1970}.
\begin{theorem}
    Let $d\geq0$ and let $\mu_1\geq \cdots \geq \mu_N$ be the eigenvalues of the adjacency matrix. Then,
    \begin{equation}\label{eq:d-improperA}
        \chi^d\bigl(G\bigr) \geq \frac{\mu_1-\mu_N}{d-\mu_N}.
    \end{equation}
\end{theorem}
We  now adapt the proof strategy from Section \ref{sec:d-proper} to derive a similar spectral bound for $d$-improper coloring of graphs.
\begin{theorem}\label{thm:d-improper}
    Let $G$ be a graph, let $d\geq 0$, and let $\overline{\deg}$ denote the average degree of $G$. Then,
    \begin{equation}\label{eq:d-improperNL}
        \chi^d \geq \frac{\lambda_N}{\lambda_N-1+d/\overline{\deg}}.
    \end{equation}
    Moreover, the bound is sharp. 
\end{theorem}
\begin{proof}
    Let $V_1,\ldots,V_{\chi^d}$ be the coloring classes of a $d$-improper coloring. Then, analogously  to the beginning of the proof of Theorem \ref{thm:d-proper}, we have
    \begin{align*}
    0 &= \RQ\bigl(\boldsymbol 1\bigr) = \frac{\sum_{v\sim w}\bigl(\boldsymbol 1(v)-\boldsymbol 1(w)\bigr)^2}{\sum_{v\in V}\deg v\cdot \boldsymbol 1(v)^2}\\
        &= 1 - 2\frac{\sum_{v\sim w}\boldsymbol 1(v)\boldsymbol 1(w)}{\sum_{v\in V}\deg v}\\
       &= 1-2\frac{\sum_{i\neq j}\sum_{\substack{v\sim w\\ v\in V_i\\w\in V_j}}1}{\frac1{\chi^d-1}\sum_{i\neq j}\sum_{v\in V_i\cup V_j}\deg v}-2\frac{\bigl(\chi^d-1\bigr)\sum_i\sum_{\substack{v\sim w\\ v,w\in V_i}}1}{\bigl(\chi^d-1\bigr)\sum_{v\in V}\deg v}
    \end{align*}    
    By substituting the coefficient $2$ in front of the second fraction by $2\chi^d-2(\chi^d-1)$, we obtain
    \begin{align}
       \nonumber  0  &= 1-2\bigl(\chi^d-1\bigr)\frac{\sum_{i\neq j}\sum_{\substack{v\sim w\\ v\in V_i\\ w\in V_j}}1}{\sum_{i\neq j}\sum_{v\in V_i\cup V_j}\deg v} +2\bigl(\chi^d-1\bigr)\frac{\bigl(\chi^d-1\bigr)\sum_i\sum_{\substack{v\sim w\\ v,w\in V_i}}1}{\bigl(\chi^d-1\bigr)\sum_{v\in V}\deg v} -2\chi^d\frac{\sum_i\sum_{\substack{v\sim w\\v,w\in V_i}}1}{\sum_{v\in V}\deg v}\\
       \label{eq:fractions2}    &= 1-2\bigl(\chi^d-1\bigr)\frac{\sum_{i\neq j}\biggl(\sum_{\substack{v\sim w\\ v\in V_i\\ w\in V_j}}1 - \sum_{\substack{v\sim w\\ v,w\in V_i \text{ or }\\ v,w\in V_j}}1\biggr)}{\sum_{i\neq j}\sum_{v\in V_i\cup V_j}\deg v} -2\chi^d \frac{\sum_i\sum_{\substack{v\sim w\\v,w\in V_i}}1}{\sum_{v\in V}\deg v}.
    \end{align}
    We now bound the first and second fraction in \eqref{eq:fractions2} separately. For the first fraction, we 
    may assume without loss of generality that $\RQ\bigl(g_{12}\bigr) = \max_{i\neq j}\RQ\bigl(g_{ij}\bigr)$. Then,
    analogously to the proof of Theorem \ref{thm:d-proper}, we have that
    \begin{align*}
        2\bigl(\chi^d-1\bigr)\frac{\sum_{i\neq j}\biggl(\sum_{\substack{v\sim w\\ v\in V_i\\ w\in V_j}}1 - \sum_{\substack{v\sim w\\ v,w\in V_i \text{ or }\\ v,w\in V_j}}1\biggr)}{\sum_{i\neq j}\sum_{v\in V_i\cup V_j}\deg v}
        &\leq 2\bigl(\chi^d-1\bigr)\frac{\biggl(\sum_{\substack{v\sim w\\ v\in V_1\\ w\in V_2}}1 - \sum_{\substack{v\sim w\\ v,w\in V_1 \text{ or }\\ v,w\in V_2}}1\biggr)}{\sum_{v\in V_1\cup V_2}\deg v} \\
        &= \bigl(\RQ\bigl(g_{12}\bigr)-1\bigr)\bigl(\chi^d-1\bigr)\\&\leq \bigl(\lambda_N-1\bigr)\bigl(\chi^d-1\bigr).
    \end{align*}
    To bound the second fraction in \eqref{eq:fractions2}, we observe that 
    \begin{align*}
        2\chi^d \frac{\sum_i\sum_{\substack{v\sim w\\v,w\in V_i}}1}{\sum_{v\in V}\deg v} &= 2\chi^d \frac{\sum_{v\in V} \sum_{\substack{w\sim v\colon\\ w\text{ in same colo-}\\\text{ring class as $v$}}}1}{2\cdot |E|}\\
        &= \frac{\chi^d}{|E|}\cdot \sum_{v\in V} \sum_{\substack{w\sim v\colon\\ w\text{ in same colo- }\\\text{ring class as $v$}}}1\\
        &\leq \frac{\chi^d}{|E|}\cdot |V| \cdot \frac d2,
    \end{align*}
    since every $v\in V$ has at most $d$ neighbors in the same coloring class, and each edge is counted twice. Hence,
    \begin{align*}
        2\chi^d \frac{\sum_i\sum_{\substack{v\sim w\\v,w\in V_i}}1}{\sum_{v\in V}\deg v}
        &\leq \frac{\chi^d}{|E|}\cdot |V| \cdot \frac d2\\
        &= \chi^d\cdot \frac{d}{\overline{\deg}}.
    \end{align*}
    Combining \eqref{eq:fractions2} with the two bounds above, we obtain
    \begin{align*}
        0 &\geq 1-\bigl(\lambda_N-1\bigr)\bigl(\chi^d-1\bigr) - \chi^d\cdot \frac{d}{\overline{\deg}}.
    \end{align*}
    Rearranging yields
    \begin{align*}
        \chi^d &\geq \frac{\lambda_N}{\lambda_N-1+d/\overline{\deg}}.
    \end{align*}
    Finally, the bound is sharp: Guo, Kang and Zwaneveld \cite{guo2024spectral} showed that the bound in \eqref{eq:d-improperA}, which coincides with our bound in the case of regular graphs, is attained with equality for a family of regular graphs.
\end{proof}
We now establish the following theorem concerning graphs that attain equality in Theorem \ref{thm:d-improper}.

\begin{theorem}\label{thm:d-improperequality}
    Let $d\geq0$ and let $G$ be a graph for which the bound in \eqref{eq:d-improperNL} is attained with equality. Then, $G$ must be a regular graph. 
\end{theorem}
\begin{proof}
    Let $V_1,\ldots,V_{\chi^d}$ be the coloring classes with respect to some $d$-improper $\chi^d$-coloring. We define \[e(v,W)\coloneqq \bigl|\bigl\{\{v,w\}\in E\colon w\in W\bigr\}\bigr|.\]
    Since the bound in \eqref{eq:d-improperNL} is attained with equality, it follows from the proof of Theorem \ref{thm:d-improper} that every vertex has exactly $d$ neighbors within its own coloring class. Furthermore, by the min-max principle \eqref{min-max2}, the functions $g_{ij}$ from Definition \ref{def:gij} must be eigenfunctions with corresponding eigenvalue
    \[
    \lambda_N = \frac{\bigl(1-d/\overline{\deg}\bigr)\chi^d}{\chi^d-1}.
    \]
    Now let $v\in V$, and let $i$ be such that $v\in V_i$. Let $j$ and $k$ be distinct from $i$ and from each other. Then,
    \begin{align*}
        0 &= Lg_{jk}(v) = \sum_{\substack{w\in V_k\\ w\sim v}}\frac1{\deg v} - \sum_{\substack{w\in V_j\\ w\sim v}}\frac1{\deg v} \\
        &= \frac1{\deg v}\bigl(e\bigl(v,V_k\bigr) - e\bigl(v,V_j\bigr)\bigr).
    \end{align*}
    Hence,
    \begin{align*}
        e(v,V_j) &=  \frac{\deg v - d}{\chi^d-1}.
    \end{align*}
    Furthermore,
    \begin{align*}
        \frac{\bigl(1-d/\overline{\deg}\bigr)\chi^d}{\chi^d-1} &= Lg_{ij}(v)\\
        &= 1+ \sum_{\substack{w\in V_j\\ w\sim v}}\frac1{\deg v} - \sum_{\substack{w\in V_i\\ w\sim v}}\frac1{\deg v}\\
        &= 1+\frac{e(v,V_j) - e(v,V_i)}{\deg v}\\
        &= 1+\frac{e(v,V_j)-d}{\deg v}.
    \end{align*}
    It follows that
    \begin{align*}
        \frac{1-\chi^d\cdot d/\overline{\deg}}{\chi^d-1} &= \frac{e(v,V_j)-d}{\deg v}.
    \end{align*}
    Rewriting gives
    \begin{align*}
        e(v,V_j) = \deg v \cdot \frac{1-\chi^d\cdot d/\overline{\deg}}{\chi^d-1} + d.
    \end{align*}
    Equating the two expressions for $e(v,V_l)$ gives
    \begin{align*}
        \deg v - d &= \deg v\bigl(1-\chi^d\cdot d/\overline{\deg}\bigr) + d\bigl(\chi^d-1\bigr),\\
        0&= d\cdot \chi^d\biggl(1-\frac{\deg v}{\overline \deg}\biggr).
    \end{align*}
    Therefore, $\deg v = \overline{\deg}$, and since $v$ was arbitrary, we conclude that $G$ must be a regular graph.
\end{proof}
By Theorem \ref{thm:d-improperequality}, the graphs for which \eqref{eq:d-improperNL} holds with equality form a subset of those for which \eqref{eq:d-improperA} is an equality (namely, the regular graphs). Theorem 10 from \cite{guo2024spectral} describes properties of such graphs.
\begin{remark}
    A possible (more intuitive) reason why the bound in \eqref{eq:d-improperNL} cannot be sharp unless $G$ is regular, is that the number of edges cannot be recovered from the normalized Laplacian spectrum of a graph, whereas it can be recovered from its adjacency spectrum. 
\end{remark}

\section{Edge-coloring hypergraphs}\label{sec:edgecol}

In this final section, we shall shine a spectral light on edge-colorings of hypergraphs (Definition \ref{def:edge-coloring}). Here, edges rather than vertices are assigned colors, under the constraint that intersecting edges must receive distinct colors. \newline

Throughout this section, we fix an oriented hypergraph $\Gamma=(V,E,\varphi)$ on $N$ nodes and $M$ edges, and we let $\chi'\coloneqq \chi'(\Gamma)$ denote its 
\emph{strong edge coloring number} (Definition \ref{def:edge-coloring}).\newline 

We shall now define the \emph{edge Laplacian}, which was introduced by Jost and Mulas (2019) \cite{JostMulas2019}.
\begin{definition}\label{def:edgeLaplacian}

    Let $\mathcal I$ and $D$ be the incidence matrix and degree matrix of $\Gamma$ from Definition \ref{def:matrices}, respectively. Then, the \emph{edge Laplacian} of $\Gamma$ is the $M\times M$ matrix
    \[
    L^1 \coloneqq \mathcal I^\top  D^{-1}\mathcal I.
    \]
    We denote its eigenvalues by
    \[
    0 \leq \mu_1\leq \cdots\leq \mu_M.
    \]
\end{definition}
\begin{remark}\label{rmk:multiplicity0}
    Recall that $L=D^{-1}\mathcal I\mathcal I^\top$. Since matrices of the form $AB$ and $BA$ have the same non-zero spectrum, this applies in particular to $L$ and $L^1$. Hence, as  shown in \cite[Corollary 14]{JostMulas2019}, if we let $m_V$ and $m_E$ be the multiplicity of the eigenvalue $0$ of $L$ and $L^1$, respectively, then
    \begin{equation}\label{eq:multiplicity0}
        m_V-m_E = |V| - |E|.
    \end{equation}
    In particular, if $|E|>|V|$, then $\mu_1=0$.
    
    Moreover, note that $$\mu_1\leq \frac N M \text{ and }\mu_N\geq \frac NM,$$
    since $\sum_{i=1}^M \mu_i = N$.
\end{remark}
\begin{example}
    Let $H_{p,k}^c$ be the $c$-uniform hyperflower with $p$ petals. Then, it follows from \eqref{eq:spectrumhyperflower} that the spectrum of $L^1$ equals
    \[
    \biggl\{ c^{(1)}, \bigl(c-k\bigr)^{(p-1)}\biggr\},
    \]
    which consists of the first $M$ non-zero eigenvalues of $L$.
\end{example}
Also for $L^1$, we have a characterization of $\RQ(\gamma)$, for functions $\gamma\colon E\to \R$:
\begin{equation}
    \RQ(\gamma) = \frac{\langle L^1 \gamma,\gamma\rangle}{\langle \gamma,\gamma\rangle} = \frac{\sum_{v\in V}\frac1{\deg v}\cdot \biggl(\sum_{e\ni v\colon v \text{ input }}\gamma(e)-\sum_{e\ni v\colon v \text{ output }}\gamma(e)\biggr)^2}{\sum_{e\in E}\gamma(e)^2}.
\end{equation}
We now prove a theorem connecting the eigenvalues of the edge Laplacian with the strong edge coloring number.

\begin{theorem}\label{thm:edgecoloring}
   Let $\Gamma$ be a $c$-uniform hypergraph in which all vertices are inputs in every edge they belong to, and let $\overline{\deg}$ denote the average degree of $\Gamma$. Then,
    \begin{equation}\label{eq:edgecoloring}
        \chi' \geq \frac{c-\mu_1}{c/\overline{\deg}-\mu_1},
    \end{equation}
    and the inequality is sharp.
\end{theorem}

We need the following:
\begin{convention}\label{conv:2}
We use the convention that, if $\mu_1 = \frac{c}{\overline{\deg}} = \frac NM$ (cf. Remark \ref{rmk:multiplicity0}), then
\[
\frac{c-\mu_1}{c/\overline{\deg}-\mu_1} = 1.
\]
Note that, for $c$-uniform hypergraphs, this situation can only occur if $\mu_1=\mu_N=c$, which means that $\Gamma$ is given by $N/c$ disjoint edges of size $c$.
    
\end{convention}

\begin{proof}
    Let $V_1,\ldots, V_{\chi'}$ be the coloring classes with respect to some proper strong edge coloring. For distinct $i,j\in\{1,\ldots,\chi'\}$, define the function $\gamma_{ij}\colon E\to\R$ as follows:
    \begin{equation}\label{eq:gammaij}
    \gamma_{ij}(e)\coloneqq \begin{cases}
        1, &\text{ if } e\in V_i,\\
        -1,&\text{ if }e\in V_j,\\
        0, &\text{ otherwise.}
    \end{cases}
    \end{equation}
    Note that the all-ones function $\boldsymbol 1$ is an eigenfunction of $L^1(\Gamma)$ with eigenvalue $c$. Hence,
    \begin{align*}
        c &= \RQ(\boldsymbol 1) = \frac{\sum_{v\in V}\frac 1{\deg v}\biggl(\sum_{e\ni v}\boldsymbol 1(e)\biggr)^2}{\sum_{e\in E}\boldsymbol 1(e)^2}\\
        &= \frac{\sum_{e\in E}\sum_{e'\in E}\biggl(\sum_{\substack{v\in V\\v\in e'\cap e}}\frac1{\deg v}\biggr)}{M}\\
        &= \frac1{M}\biggl( \sum_{v\in V}\sum_{e\ni v}\frac1{\deg v} + 2\sum_{i\neq j}\sum_{e\in V_i}\sum_{e'\in V_j}\biggl(\sum_{\substack{v\in V\\ v\in e\cap e'}}\frac1{\deg v}\biggr)\biggr)\\
        &= \frac{N}{M} + 2\frac{\sum_{i\neq j}\sum_{e\in V_i}\sum_{e'\in V_j}\biggl(\sum_{\substack{v\in V\\ v\in e\cap e'}}\frac1{\deg v}\biggr)}{\frac1{\chi'-1}\cdot \sum_{i\neq j}\sum_{e\in V_i\cup V_j}1}\\
        &\overset{\text{WLOG}}{\leq} \frac{N}{M} + 2(\chi'-1) \frac{\sum_{e\in V_1}\sum_{e'\in V_2}\biggl(\sum_{\substack{v\in V\\ v\in e\cap e'}}\frac1{\deg v}\biggr)}{\sum_{e\in V_1\cup V_2}1}\\
        &\overset{(i)}{=} \frac{N}{M} + (\chi'-1)\biggl(\frac{N}{M}-\RQ(\gamma_{12})\biggr)\\
        &\overset{(ii)}{\leq} \chi' \cdot \frac{N}{M}-(\chi'-1)\cdot \mu_1,
    \end{align*}
    where in $(i)$ we used the fact that
    \[
    \RQ(\gamma_{12}) = \frac NM - 2\frac{\sum_{e\in V_1}\sum_{e'\in V_2}\biggl(\sum_{\substack{v\in V\\ v\in e\cap e'}}\frac1{\deg v}\biggr)}{\sum_{e\in V_1\cup V_2}1},
    \]
    while in $(ii)$ we use the min-max principle from \eqref{min-max2}.
    
    The inequality we found is equivalent to
    \[
    \chi'\geq \frac{c-\mu_1}{N/M-\mu_1} = \frac{c-\mu_1}{c/\overline{\deg}-\mu_1},
    \]
    provided that $c/\overline{\deg}-\mu_1>0$; otherwise, the inequality is trivial (by Convention \ref{conv:2}).

    Finally, to show that the inequality is sharp, consider $c$-uniform hyperflowers $H_{p,k}^c$ with $p$ petals and $k$ central vertices. In this case, it follows from \eqref{eq:spectrumhyperflower} that $\mu_1=c-k$, and it is easy to see that $\chi' = p=M$. Furthermore,
    \[
    \frac{c-\mu_1}{N/M-\mu_1} = \frac{k}{(M(c-k)+k)/M-(c-k)} = \frac{k}{k/M} = M,
    \]
    hence the bound is attained with equality.
\end{proof}
We have the following proposition about the equality case in Theorem \ref{thm:edgecoloring}. 
\begin{proposition}\label{prop:edgecoloring}
    Let $\Gamma$ be a $c$-uniform hypergraph in which all vertices are inputs in all edges they belong to, and for which the bound in \eqref{eq:edgecoloring} is attained with equality. Then, we have the following two cases:
    \begin{enumerate}
        \item If $\mu_1=0$, then $\Gamma$ is $k$-regular for some $k$, and $\chi' = k$. Furthermore, the multiplicity of the eigenvalue $0$ is at least $k-1$.
        \item If $\mu_1>0$, then $\mu_1$ is an eigenvalue of $L^1$ and $L$ with multiplicity at least $\chi'-1$.
    \end{enumerate}
\end{proposition}
\begin{proof}
    For the first case, assume that $\mu_1=0$. Then, the bound from \eqref{eq:edgecoloring} becomes
    \[
    \chi'\geq \overline{\deg}.
    \]
    Since we have the trivial bound $\chi'\geq \max\deg$, it must hold that $\overline \deg = \max\deg$, which implies that $\Gamma$ is $\overline \deg$-regular, and $\chi' = \overline{\deg}$.

    Furthermore, if $\mu_1=0$, then the functions $\gamma_{1i}$ for $i=1,\ldots,\chi'$ from \eqref{eq:gammaij} are linearly independent eigenfunctions with eigenvalue $\mu_1=0$. Therefore, the multiplicity of $0$ is at least $\chi'-1=k-1$.

    For the second case, assume that $\mu_1>0$. Also in this case, for $j=2,\ldots,\chi'$, the functions $\gamma_{1j}$ from \eqref{eq:gammaij} are linearly independent eigenfunctions of $L^1$ with eigenvalue $\mu_1$. Hence, the functions $D^{-1}\mathcal I\gamma_{ij}$ are $\chi'-1$ linearly independent eigenfunctions of $L$ with eigenvalue $\mu_1$.
\end{proof}
We end with one final remark.
\begin{remark}
    Let $\Gamma$ be a $c$-uniform hypergraph in which all vertices are inputs in all edges they belong to.
    If $M>N$, then $\mu_1=0$, and the bound from \eqref{eq:edgecoloring} becomes $\chi'\geq\overline{\deg}$, which does not provide any additional information beyond the trivial bound $\chi'\geq \max\deg$. However, if $M<N+\chi'-1$, then $0$ is an eigenvalue of $L$ with multiplicity at least $N+\chi'-1-M$.
\end{remark}

\section*{Funding}
Raffaella Mulas is supported by the Dutch Research Council (NWO) through the grant VI.Veni.232.002.

\color{black}
\bibliographystyle{plain} 

\bibliography{Bibliography}

\begin{thebibliography}{10}

\bibitem{Abiad20}
A.~Abiad, R.~Mulas, and D.~Zhang.
\newblock {Coloring the normalized Laplacian for oriented hypergraphs}.
\newblock {\em Linear Algebra and Its Applications}, 629:192--207, 2021.

\bibitem{adamson2024distributed}
Duncan Adamson, Will Rosenbaum, and Paul~G Spirakis.
\newblock Distributed weak independent sets in hypergraphs: Upper and lower bounds.
\newblock {\em arXiv preprint arXiv:2411.13377}, 2024.

\bibitem{AndreottiMulas2022}
E.~Andreotti and R.~Mulas.
\newblock Signless normalized {L}aplacian for hypergraphs.
\newblock {\em Electron. J. Graph Theory Appl. (EJGTA)}, 10(2):485--500, 2022.

\bibitem{Banerjee2021}
Anirban Banerjee.
\newblock On the spectrum of hypergraphs.
\newblock {\em Linear Algebra Appl.}, 614:82--110, 2021.

\bibitem{Berge1970}
Claude Berge.
\newblock {\em Graphes et hypergraphes}, volume No. 37 of {\em Monographies Universitaires de Math\'ematiques}.
\newblock Dunod, Paris, 1970.

\bibitem{Bilu2006}
Yonatan Bilu.
\newblock Tales of {H}offman: three extensions of {H}offman's bound on the graph chromatic number.
\newblock {\em J. Combin. Theory Ser. B}, 96(4):608--613, 2006.

\bibitem{ChenLiuRobinsonRusnakWang2018}
Gina Chen, Vivian Liu, Ellen Robinson, Lucas~J. Rusnak, and Kyle Wang.
\newblock A characterization of oriented hypergraphic {L}aplacian and adjacency matrix coefficients.
\newblock {\em Linear Algebra Appl.}, 556:323--341, 2018.

\bibitem{ChenRaoRusnakYang2015}
Vinciane Chen, Angeline Rao, Lucas~J. Rusnak, and Alex Yang.
\newblock A characterization of oriented hypergraphic balance via signed weak walks.
\newblock {\em Linear Algebra Appl.}, 485:442--453, 2015.

\bibitem{chung}
F.~Chung.
\newblock {\em Spectral graph theory}.
\newblock American Mathematical Soc., 1997.

\bibitem{chung1998erdos}
F.~Chung and R.~Graham.
\newblock {\em {Erdős on graphs: His legacy of unsolved problems}}.
\newblock CRC Press, 1998.

\bibitem{Gabriel-colouring}
G.~Coutinho, R.~Grandsire, and C.~Passos.
\newblock Colouring the normalized laplacian.
\newblock {\em Electronic Notes in Theoretical Computer Science}, 346:345–354, August 2019.

\bibitem{DeWerra1979}
D.~de~Werra.
\newblock Regular and canonical colorings.
\newblock {\em Discrete Math.}, 27(3):309--316, 1979.

\bibitem{DuttweilerReff2019}
Luke Duttweiler and Nathan Reff.
\newblock Spectra of cycle and path families of oriented hypergraphs.
\newblock {\em Linear Algebra Appl.}, 578:251--271, 2019.

\bibitem{ElphickWocjan2015}
C.~Elphick and P.~Wocjan.
\newblock Unified spectral bounds on the chromatic number.
\newblock {\em Discuss. Math. Graph Theory}, 35(4):773--780, 2015.

\bibitem{Erd94}
P.~Erdős.
\newblock Problems and results on set systems and hypergraphs.
\newblock {\em Extremal problems for finite sets (Visegr{\'a}d, 1991)}, 3:217--227, 1994.

\bibitem{ErdosHajnal1966}
P.~Erdős and A.~Hajnal.
\newblock On chromatic number of graphs and set-systems.
\newblock {\em Acta Math. Acad. Sci. Hungar.}, 17:61--99, 1966.

\bibitem{FurediKahn1986}
Z.~F\"uredi and J.~Kahn.
\newblock On the dimensions of ordered sets of bounded degree.
\newblock {\em Order}, 3(1):15--20, 1986.

\bibitem{GrillietteReynesRusnak2022}
Will Grilliette, Josephine Reynes, and Lucas~J. Rusnak.
\newblock Incidence hypergraphs: injectivity, uniformity, and matrix-tree theorems.
\newblock {\em Linear Algebra Appl.}, 634:77--105, 2022.

\bibitem{guo2024spectral}
Krystal Guo, Ross~J Kang, and Gabri{\"e}lle Zwaneveld.
\newblock Spectral approaches for $ d $-improper chromatic number.
\newblock {\em arXiv preprint arXiv:2411.06941}, 2024.

\bibitem{Hoffman1970}
Alan~J. Hoffman.
\newblock On eigenvalues and colorings of graphs.
\newblock In {\em Graph {T}heory and its {A}pplications ({P}roc. {A}dvanced {S}em., {M}ath. {R}esearch {C}enter, {U}niv. of {W}isconsin, {M}adison, {W}is., 1969)}, pages 79--91. Academic Press, New York-London, 1970.

\bibitem{JostMulas2019}
J.~Jost and R.~Mulas.
\newblock Hypergraph {L}aplace operators for chemical reaction networks.
\newblock {\em Adv. Math.}, 351:870--896, 2019.

\bibitem{KitouniReff2019}
Ouail Kitouni and Nathan Reff.
\newblock Lower bounds for the {L}aplacian spectral radius of an oriented hypergraph.
\newblock {\em Australas. J. Combin.}, 74:408--422, 2019.

\bibitem{Mulas-Cheeger}
R.~Mulas.
\newblock {A Cheeger Cut for Uniform Hypergraphs}.
\newblock {\em Graphs and Combinatorics}, 37:2265--2286, 2021.

\bibitem{Mulas2021}
R.~Mulas.
\newblock Sharp bounds for the largest eigenvalue.
\newblock {\em Math. Notes}, 109(1-2):102--109, 2021.

\bibitem{mulas2022hypergraphs}
R.~Mulas, D.~Horak, and J.~Jost.
\newblock Graphs, simplicial complexes and hypergraphs: Spectral theory and topology.
\newblock In {\em Higher-order systems}, pages 1--58. Springer, 2022.

\bibitem{MulasZhang2021}
R.~Mulas and D.~Zhang.
\newblock Spectral theory of {L}aplace operators on oriented hypergraphs.
\newblock {\em Discrete Math.}, 344(6):Paper No. 112372, 22, 2021.

\bibitem{nikiforov2019hoffman}
Vladimir Nikiforov.
\newblock Hoffman's bound for hypergraphs.
\newblock {\em arXiv preprint arXiv:1908.01433}, 2019.

\bibitem{Reff2014}
Nathan Reff.
\newblock Spectral properties of oriented hypergraphs.
\newblock {\em Electron. J. Linear Algebra}, 27:373--391, 2014.

\bibitem{Reff2016}
Nathan Reff.
\newblock Intersection graphs of oriented hypergraphs and their matrices.
\newblock {\em Australas. J. Combin.}, 65:108--123, 2016.

\bibitem{ReffRusnak2012}
Nathan Reff and Lucas~J. Rusnak.
\newblock An oriented hypergraphic approach to algebraic graph theory.
\newblock {\em Linear Algebra Appl.}, 437(9):2262--2270, 2012.

\bibitem{Rusnak2013}
Lucas~J. Rusnak.
\newblock Oriented hypergraphs: introduction and balance.
\newblock {\em Electron. J. Combin.}, 20(3):Paper 48, 29, 2013.

\bibitem{RusnakReynesJohnsonYe2021}
Lucas~J. Rusnak, Josephine Reynes, Skyler~J. Johnson, and Peter Ye.
\newblock Generalizing {K}irchhoff laws for signed graphs.
\newblock {\em Australas. J. Combin.}, 81:388--411, 2021.

\bibitem{RusnakReynesLiYanYu2024}
Lucas~J. Rusnak, Josephine Reynes, Russell Li, Eric Yan, and Justin Yu.
\newblock The determinant of \{{$\pm$}1\}-matrices and oriented hypergraphs.
\newblock {\em Linear Algebra Appl.}, 702:161--178, 2024.

\bibitem{RusnakRobinsonSchmidtShroff2019}
Lucas~J. Rusnak, Ellen Robinson, Martin Schmidt, and Piyush Shroff.
\newblock Oriented hypergraphic matrix-tree type theorems and bidirected minors via {B}oolean order ideals.
\newblock {\em J. Algebraic Combin.}, 49(4):461--473, 2019.

\bibitem{shi1992signed}
Chuan-Jin Shi.
\newblock A signed hypergraph model of the constrained via minimization problem.
\newblock {\em Microelectronics journal}, 23(7):533--542, 1992.

\end{thebibliography}

\end{document}